\documentclass[a4paper,12pt]{amsart}

\usepackage{float}
\usepackage{euscript,eufrak,verbatim}
\usepackage{graphicx}
\usepackage{amscd,pifont}
\usepackage[usenames]{color}
\usepackage[colorlinks,linkcolor=red,anchorcolor=blue,citecolor=blue]{hyperref}
\usepackage{bbm}
\usepackage{amsmath}
\usepackage{amsthm}
\usepackage[all]{xy}
\usepackage{graphicx}

\newtheorem{thm}{Theorem}[section]
\newtheorem{prop}[thm]{Proposition}
\newtheorem{df}[thm]{Defintion}
\newtheorem{lem}[thm]{Lemma}
\newtheorem{cor}[thm]{Corollary}

\newtheorem{ex}[thm]{Example}
\newtheorem{exs}[thm]{Examples}
\newtheorem{conj}[thm]{Conjecture}

\def\N{\mathbb{N}}

\def\Ch{{\rm (Ch)}}
\def\S{{\rm (S)}}
\def\So{{\rm (S_0)}}
\def\Z{\mathbb{Z}}
\def\N{\mathbb{N}}
\def\R{\mathbb{R}}

\def\GS{{\rm (GS)}}
\def\GSo{{\rm (GS_0)}}

\makeatletter 
\@addtoreset{equation}{section}
\makeatother  

\title[Construction of some Chowla sequences]{Construction of some Chowla sequences}
\author{Ruxi Shi 
}
\address
{LAMFA, UMR 7352 CNRS, Universit\'e de Picardie,
	33 rue Saint Leu, 80039 Amiens, France}
\email{ruxi.shi@u-picardie.fr}
\begin{document}
	
	\maketitle

\begin{abstract}
	For numerical sequences taking values $0$ or complex numbers of modulus $1$, we define
	Chowla property and Sarnak property. We prove that Chowla property implies Sarnak property. 
	We also prove that for Lebesgue almost every $\beta>1$, the sequence $(e^{2\pi \beta^n})_{n\in \N}$ shares Chowla
	property and consequently is orthogonal to all topological dynamical systems of zero entropy. It is also
	discussed whether the samples of a given random sequence have Chowla property almost surely. Some
	dependent random sequences having almost surely Chowla property are constructed. 
\end{abstract}

\section{Introduction}

Recall that the M\"obius function is an arithmetic function defined by the formula
$$
\mu(n)=
\begin{cases}
1 &\text{if}~n=1,\\
(-1)^k ~&\text{if $n$ is a product of $k$ distinct primes,}\\
0 &\text{otherwise}.
\end{cases}
$$
It is well known that the M\"obius function plays an important role in number theory: the statement 
$$
\sum_{n=1}^{N} \mu(n)=O_{\epsilon}(N^{\frac{1}{2}+\epsilon}), ~\text{for each}~\epsilon>0,
$$ 
is equivalent to Riemann's hypothesis (see \cite{Tit1986}).

Chowla \cite{Cho1965} formulated a conjecture on the higher orders correlations of the M\"obius function: 

\begin{conj}\label{Conj: Chowla}
	For each choice of $0\le a_1< \dots <a_r, r\ge 0, i_s\in \{1,2\}$ not all equal to $2$,
	$$
	\lim\limits_{N\to \infty}\frac{1}{N}\sum_{n=0}^{N-1} \mu^{i_0}(n+a_1)\cdot \mu^{i_1}(n+a_2)\cdot \dots \cdot \mu^{i_r}(n+a_r)=0.
	$$
\end{conj}

Let $T$ be a continuous transformation on a compact metric space $X$. Following Sarnak \cite{Sar2011}, we will say a sequence $(a(n))_{n\in \N}$ of complex numbers is orthogonal to the topological dynamical system $(X,T)$, if 
\begin{equation}\label{eq:Sarnak conjecture}
\lim\limits_{N\to \infty}\frac{1}{N}\sum_{n=0}^{N-1} f(T^nx)a(n)=0,
\end{equation}
for all $x\in X$ and for all $f\in C(X)$ where $C(X)$ is the space of continuous functions on $X$. 

In 2010, Sarnak \cite{Sar2011} formulated the following conjecture.

\begin{conj}\label{Conj: Sarnak}
	The M\"obius function is orthogonal to all the topological dynamical systems of zero entropy. 
\end{conj}
Sarnak \cite{Sar2011} also proved that Chowla's conjecture implies his conjecture.

Recently, El Abdalaoui, Kulaga-Przymus, Lemanczyk and de La Rue \cite{AbaKulLemDe2014} studied these conjectures from ergodic theory point of view. Let $z$ be an arbitrary sequence taking values in $\{-1,0,1\}$ as the M\"obius function. Following El Abdalaoui et al., we say that $z$ satisfies the condition $\Ch$ if it satisfies the condition in Chowla's conjecture; similarly, it satisfies the condition $\So$ if it satisfies the condition in Sarnak's conjecture. They proved that the condition $\Ch$ implies the condition $\So$. El Abdalaoui et al. also provided some non-trivial examples of sequences satisfying the condition $\Ch$ and the condition $\So$ respectively.

Let $(X, T)$ be a topological dynamical system. Recall that a point $x\in X$ is said to be {\em completely deterministic} if for any accumulation point $\nu$ of $\frac{1}{N}\sum_{n=0}^{N-1}\delta_{T^n x}$ where $\delta_x$ is the Dirac measure supported on $x$, the system $(X, T, \nu)$ is of zero entropy. El Abdalaoui et al. also defined the condition $\S$, formally stronger than $\So$ by requiring
\begin{equation*}
\lim\limits_{N\to \infty}\frac{1}{N}\sum_{n=0}^{N-1} f(T^nx)z(n)=0
\end{equation*}
for any topological dynamical system $(X,T)$, any $f\in C(X)$ and any completely deterministic point $x\in X$. They proved the equivalence between the condition $\S$ and the condition $\So$.



Fan \cite{F} defined that a bounded sequence $(c(n))_{n\in \N}$ is \textit{oscillating of order $d$} ($d\ge 1$ being an integer) if  
$$
\lim\limits_{N\to \infty}\frac{1}{N}\sum_{n=0}^{N-1} c(n) e^{2\pi i P(n)}=0
$$
for any $P\in \mathbb{R}_d[t]$ where $\mathbb{R}_d[t]$ is the space of real polynomials of degree smaller than or equal to $d$; it is said to be \textit{fully oscillating} if it is oscillating of order $d$ for every integer $d\ge 1$. It is clear that an oscillating sequence of order $d$ is always oscillating of order $d-1$ for any $d\ge 2$. In \cite{Shi2017}, we gave several equivalent definitions of oscillating sequences, one of which states that a sequence is fully oscillating if and only if it is orthogonal to all affine maps on all compact abelian groups of zero entropy. This result yields that sequences satisfying the condition $\S$ are fully oscillating. 

In what follows, a sequence taking values in $\{-1,0,1\}$ is said to have \textit{Chowla property} (resp. \textit{Sarnak property}) if it satisfies the condition $\Ch$ (resp. the condition $\S$). We consider a more general class of sequences, taking values $0$ or complex numbers of modulus $1$. We define Chowla property and Sarnak property for this class of sequences. We prove that Chowla property implies Sarnak property.
A sequence having Chowla property is called a \textit{Chowla sequence}. We are interested in finding Chowla sequences. We prove that for almost every $\beta>1$, the sequence $(e^{2\pi i \beta^n})_{n\in \N}$  is a Chowla sequence and consequently is orthogonal to all topological dynamical systems of zero entropy. This extends a result of Akiyama and Jiang \cite{AkiJia2016} who proved that for almost every $\beta>1$, the sequence $(e^{2\pi i \beta^n})_{n\in \N}$ is fully oscillating. On the other hand, from a probability point of view, we discuss whether the samples of a given random sequence have almost surely Chowla property. We also construct a sequence of dependent random variables which is a Chowla sequence almost surely.

The paper is organized as follows. In Section \ref{Sec: Preliminary}, we will recall some basic notions in topological and measurable dynamical systems. In Section \ref{Sec: Defitions}, we will introduce the definition of Chowla property for sequences taking values $0$ or complex numbers of modulus $1$ and the definition of Sarnak property for sequences of complex numbers.  In Section \ref{Sec: examples}, it will be proved that for almost every $\beta>1$, the sequence $(e^{2\pi i\beta^n})_{n\in \N}$ is a Chowla sequence.  In Section \ref{Sec: Ch implies S}, we will prove that Chowla property implies Sarnak property. In Section \ref{Sec: Ch random variable}, we will study random sequences: we will give an equivalent condition for stationary random sequences to be Chowla sequences almost surely; we will construct sequences of dependent random variables which are Chowla sequences almost surely;  we will also prove that Sarnak property does not imply Chowla property.


\medskip
Through this paper, we denote by $\mathbb{N}=\{0, 1, 2, \dots\}$ and $\mathbb{N}_{\ge 2}=\mathbb{N}\setminus \{0,1\}$.

\section{Preliminary}\label{Sec: Preliminary}
In this section, we recall some basic notions in dynamical systems. 

\subsection{Quasi-generic measure}
Let $(X, T)$ be a topological dynamical system. Denote by $\mathcal{P}_T(X)$ the set of all $T$-invariant probability measures on $X$ equipped with the weak topology. Under this weak topology, it is well known that the space $\mathcal{P}_T(X)$ is compact, convex and not empty. Let $x\in X$.
Set 
$$
\delta_{N,T,x}:=\frac{1}{N}\sum_{n=0}^{N-1} \delta_{T^nx}.
$$ 
Sometimes, we write simply $\delta_{N,x}$ if the system is fixed.
Denote the set of all quasi-generic measures associated to $x$ by
$$
Q\text{-}gen(x):=\{\nu\in\mathcal{P}_T(X): \delta_{N_k,x}\xrightarrow{k\to \infty}\nu, (N_k)_{k\in \N}\subset \mathbb{N} \}.
$$
By compactness of $\mathcal{P}_T(X)$, the set $Q\text{-}gen(x)$ is not empty. Furthermore, the point $x$ is said to be \textit{completely deterministic} if for any measure $\nu \in Q\text{-}gen(x)$, the entropy $h(T,\nu)$ is zero.
On the other hand, we say that the point $x$ is \textit{quasi-generic} for a measure $\nu\in \mathcal{P}_T(X)$ along $(N_k)_{k\in \N}$ if $\delta_{N_k,x}\rightarrow \nu$ as $k\to \infty$. Moreover, $x$ is said to be \textit{generic} for $\nu$ if the set $Q\text{-}gen(x)$ contains only one element $\nu$.

\subsection{Joinings}
The notion of joinings was first introduced by Furstenberg \cite{Fur1967}. It is a very powerful tool in ergodic theory. In this subsection, we will recall some basic concepts about joinings.

By a measurable dynamical system $(Y, \mathcal{B}, S, \mu)$, we mean that $(Y,\mathcal{B})$ is a measurable space, $S$ is a measurable map from $Y$ to itself and $\mu$ is an $S$-invariant measure, that is to say, $\mu(S^{-1}B)=\mu(B)$ for any $B\in \mathcal{B}$.

Let $(X,\mathcal{A},T,\mu)$ and $(Y,\mathcal{B},S,\nu)$ be two measurable dynamical systems. We say that $(Y,\mathcal{B},S,\nu)$ is a \textit{factor} of $(X,\mathcal{A}, T,\mu)$ (or $(X,\mathcal{A}, T,\mu)$ is an \textit{extension} of $(Y,\mathcal{B},S,\nu)$) if there exists a measurable map $\pi:X\to Y$ such that $\pi\circ S=T\circ \pi$ and the pushforward of $\mu$ via $\pi$ is $\nu$, that is, $\pi_*\mu=\nu$. We call the map $\pi$ the \textit{factor map}. In this case, it is classical that we can identify $\mathcal{B}$ with the $\sigma$-algebra $\pi^{-1}(\mathcal{B})\subset \mathcal{A}$, which is $T$-invariant. On the other hand, any $T$-invariant sub-$\sigma$-algebra $\mathcal{C}\subset\mathcal{A}$ will be identified with the corresponding factor $(X/ \mathcal{C},\mathcal{C}, \mu_{|\mathcal{C}}, T)$.

Let $\pi_i: (X,\mathcal{A}, T,\mu) \to (Y_i,\mathcal{B}_i, S_i,\nu_i)$ be two factor maps for $i=1,2$. We denote by $(Y_1,\mathcal{B}_1,S_1,\nu_1)\bigvee (Y_2,\mathcal{B}_2,S_2,\nu_2)$ the smallest factor of $(X,T,\mathcal{A},\mu)$ containing both $\pi_1^{-1}(\mathcal{B}_1)$ and $\pi_2^{-1}(\mathcal{B}_2)$. It is easy to see that $(Y_1,\mathcal{B}_1,S_1,\nu_1)\bigvee (Y_2,\mathcal{B}_2,S_2,\nu_2)$ is a common extension of $(Y_1,\mathcal{B}_1,S_1,\nu_1)$ and $ (Y_2,\mathcal{B}_2,S_2,\nu_2)$.

Recall that a measurable dynamical system is said to be a \textit{Kolmogorov system} if any non-trivial factor of it has positive entropy. Let $\pi: (X,\mathcal{A}, T,\mu) \to (Y,\mathcal{B}, S, \nu)$ be a factor map. The \textit{relative entropy} of the factor map $\pi$ is defined by the quantity $h(T,\mu)-h(S,\nu)$.  We say that the factor map $\pi$ is \textit{relatively Kolmogorov} if $\pi$ is non trivial and if for any intermediate factor $(Z,\mathcal{C}, U, \rho)$ with factor map $\pi': (Z,\mathcal{C}, U, \rho) \to (Y,\mathcal{B}, S, \nu)$, the relative entropy of $\pi'$ is positive unless $\pi'$ is an isomorphism.

Given measurable dynamical systems $(X_i, \mathcal{B}, T_i, \mu_i)$ for $i=1,2,\dots, k$, we recall that a \textit{joining} $\rho$ is a $T_1\times \dots \times T_k$-invariant probability measure on $(X_1\times\dots \times X_k, \mathcal{B}_1\otimes \dots \otimes \mathcal{B}_k)$ such that $(\pi_i)_{*}(\rho)=\mu_i$, where $\pi_i$ is the projection from $X_1\times\dots \times X_k$ to $X_i$ for $1\le i\le k$. The set consisting of all joinings is denoted by $J(T_1, T_2, \dots, T_k)$. Following Furstenberg \cite{Fur1967}, $(X_1, \mathcal{B}, T_1, \mu_1)$ and $(X_2, \mathcal{B}, T_2, \mu_2)$ are said to be \textit{disjoint} if their joining is unique, which is $\mu_1\otimes \mu_2$.

Let $\pi_i: (X_i,\mathcal{A}_i, T_i,\mu_i) \to (Y,\mathcal{B}, S,\nu)$ be two factor maps for $i=1,2$. Given $\rho\in J(S,S)$, the \textit{relatively independent extension} of $\rho$ is defined by $\widehat{\rho}\in J(T_1,T_2)$ such that
$$
\widehat{\rho}(A_1\times A_2):=\int_{Y\times Y} \mathbb{E}[1_{A_1}| \pi_1^{-1}(\mathcal{B})](x) \mathbb{E}[1_{A_2}| \pi_2^{-1}(\mathcal{B})](y) d\rho(x,y),
$$
for each $A_i\in \mathcal{A}_i, i=1,2.$  We say that the factor maps $\pi_1$ and $\pi_2$ are \textit{relatively independent} over their common factor if $J(T_1, T_2)$ only consists of $\widehat{\Delta}$ where $\Delta\in J(S,S)$  is given by $\Delta(B_1\times B_2):=\nu(B_1\cap B_2)$ for any $B_1, B_2\in \mathcal{B}.$


%


The following lemma is classical.
\begin{lem}[\cite{th}, Lemma 3] \label{Relatively Independent}
	Let $\pi_i: (X_i,\mathcal{A}_i, T_i,\mu_i) \to (Y,\mathcal{B}, S,\nu)$ be two factor maps for $i=1,2$. Suppose that $\pi_1$ is of relative zero entropy and $\pi_2$ is relatively Kolmogorov. Then $\pi_1$ and $\pi_2$ are relatively independent.
\end{lem}

\subsection{Shift system}
Let $X$ be a nonempty compact metric space. The product space $X^{\mathbb{N}}$ endowed with the product topology is a compact metric space. Coordinates of $x\in X^{\N}$ will be denoted by $x(n)$ for $n\in \N$. 

On the space $X^\N$ there is a natural continuous action by the \textit{shift} $S: X^{\mathbb{N}}\to X^{\mathbb{N}}$ defined by $(x(n))_{n\in \N} \mapsto (x(n+1))_{n\in \N}$. We call the topological dynamical system $(X^\N, S)$ a \textit{(topological) shift system}. Through this paper, we always denote by $S$ the shift on the shift system.

We recall that the subsets of $X^\N$ of the form 
$$
[C_0, C_1, \dots, C_k]:=\{x\in X^\N: x(j)\in C_j,  ~\forall 0\le j \le k \}
$$
where $k\ge 0$ and $C_0, \dots, C_k$ are open subsets in $X$ , are called \textit{cylinders} and they form a basis of the product topology of $X^\N$. It follows that any Borel measure on $X^\mathbb{N}$ is uniquely determined by the values on these cylinders. Through this paper, we always denote $\mathcal{B}$ the Borel $\sigma$-algebra on the shift system. For any Borel measure $\mu$ on $X^\mathbb{N}$, we sometimes write $(X^\N, S, \mu)$ for  the measurable shift system $(X^\N, \mathcal{B}, S, \mu)$.


In what follows, we denote by $F:X^{\mathbb{N}}\to X$ the projection on the first coordinate, i.e. $(x(n))_{n\in \N} \mapsto x(0)$. It is easy to see that the projection is continuous. 

\section{Definitions of Chowla property}\label{Sec: Defitions}

In this section, we will show how to generalize Chowla property from $\{-1,0, 1\}^{\N}$ to $(S^1\cup \{0\})^{\N}$ where $S^1$ is the unit circle. Before we really get into the generalization of Chowla property, we first generalize Sarnak property from the sequences in $\{-1,0, 1\}^\N$ to ones in $\mathbb{C}^\mathbb{N}$. To this end, we define the condition $({\rm GS_0})$ for sequences of complex numbers as follows. 

\subsection{Sarnak sequences}
Here, we first give the definition of Sarnak sequences and then show several equivalent definitions of Sarnak sequences.
\begin{df}
	We say that a sequence $(z(n))_{n\in \N}$ of complex numbers satisfies the \textit{condition $({\rm GS_0})$} if for each homeomorphism $T$ of a compact metric space $X$ with topological entropy $h(T)=0$, we have 
	$$
	\lim\limits_{N\to \infty}\frac{1}{N}\sum_{n=0}^{N-1} f(T^nx)z(n)=0,
	$$ 
	for all $f\in C(X)$ and for all $x\in X$.
\end{df}

It is not hard to see that the the condition $({\rm GS_0})$ is a generalization of the condition $\So$. Similarly, we generalize the condition $\S$ from $\{-1,0, 1\}^\N$ to $\mathbb{C}^\mathbb{N}$.

\begin{df}\label{(S) Sequences }
	We say that a sequence $(z(n))_{n\in \N}$ of complex numbers satisfies the \textit{condition $\GS$} if for each homeomorphism $T$ of a compact metric space $X$,
	$$
	\lim\limits_{N\to \infty}\frac{1}{N}\sum_{n=0}^{N-1} f(T^nx)z(n)=0,
	$$
	for all $f\in C(X)$ and for all completely deterministic points $x\in X$.
\end{df}

It is observed that the condition $\GS$ implies the condition $\GSo$. Similarly to the proof of $\So \Rightarrow \S$ in \cite{AbaKulLemDe2014}, the  condition $\GSo$ implying the condition $\GS$ follows from the characterization of completely deterministic points by Weiss \cite{w}:  

\vspace{8pt}
{\em A sequence $u$ is completely deterministic if and only if, for any $\epsilon>0$ there exists K such that, after removing from u a subset of density less than $\epsilon$, what is left can be covered by collection $C$ of $K$-blocks such that $|C|< 2^{\epsilon K}$.}
\vspace{8pt}

\noindent Therefore, we obtain that the condition $\GSo$ is equivalent to the condition $\GS$. In what follows, we say that a sequence has \textit{Sarnak property} if it satisfies the condition $\GSo$ (or equivalently the condition $\GS$). Such sequence is called a  \textit{Sarnak sequence}. 


Now we show several equivalent definitions of Sarank sequences.
Following the definitions of MOMO (M\"obius Orthogonality on Moving Orbits) property and strong MOMO property in \cite{AbaLemDe2017}, we define the following. 
\begin{df}
	We say that a sequence $(z(n))_{n\in \N}$ of complex numbers is \textbf{orthogonal on moving orbits} to a dynamical system $(X,T)$ if for any increasing sequence of integers $0=b_0<b_1<b_2<\cdots$ with $b_{k+1}-b_k\to \infty$, for any sequence $(x_k)_{k\in \N}$ of points in $X$, and any $f\in C(X)$,
	\begin{equation}
	\lim\limits_{K\to \infty}\frac{1}{b_K}\sum_{k=0}^{K-1} \sum_{b_k\le b<b_{k+1}} f(T^{n-b_k}x_k)z(n)=0.
	\end{equation} 
\end{df}

It is easy to see that the a sequence orthogonal on moving orbits to a dynamical system $(X,T)$ is always orthogonal to $(X,T)$.

\begin{df}
	We say that a sequence $(z(n))_{n\in \N}$ of complex numbers is \textbf{strongly orthogonal on moving orbit} to a dynamical system $(X,T)$ if for any increasing sequence of integers $0=b_0<b_1<b_2<\cdots$ with $b_{k+1}-b_k\to \infty$, for any sequence $(x_k)_{k\in \N}$ of points in $X$, and any $f\in C(X)$,
	\begin{equation}\label{eq:SOMO}
	\lim\limits_{K\to \infty}\frac{1}{b_K}\sum_{k=0}^{K-1} \left| \sum_{b_k\le b<b_{k+1}} f(T^{n-b_k}x_k)z(n) \right|=0.
	\end{equation} 
\end{df}

Obviously, if putting $f=1$ in (\ref{eq:SOMO}), then we have that if a sequence $(z(n))_{n\in \N}$ is strongly orthogonal on moving orbits to a dynamical system $(X,T)$, then
\begin{equation}\label{eq:(5)}
\lim\limits_{K\to \infty}\frac{1}{b_K}\sum_{k=0}^{K-1} \left| \sum_{b_k\le b<b_{k+1}} z(n) \right|=0,
\end{equation}
for any increasing sequence of integers $0=b_0<b_1<b_2<\cdots$ with $b_{k+1}-b_k\to \infty$.

\begin{df}
	We say that a sequence $(z(n))_{n\in \N}$ of complex numbers is \textbf{uniformly orthogonal} to a dynamical system $(X,T)$ if for any $f\in C(X)$,  the sums $\frac{1}{N}\sum_{n=1}^{N} f(T^nx)z(n)$ is uniformly convergent to zeros for every $x\in X$.
\end{df}

By Footnote 4 of Page 4 in \cite{AbaLemDe2017}, we have that if a sequence is strongly orthogonal on moving orbits to a dynamical system $(X,T)$, then it is uniformly orthogonal to $(X,T)$.

Combining Corollary 9 and Corollary 10 in \cite{AbaLemDe2017}, we obtain directly the following equivalent definitions of Sarnak sequences.
\begin{prop}\label{prop:equivalent definitions}
	Let $(z(n))_{n\in \N}$ be a sequence of complex numbers. The following are equivalent.
	\begin{itemize}
		\item [(1)] The sequence $(z(n))_{n\in \N}$ is a Sarnak sequence.
		\item [(2)] The sequence $(z(n))_{n\in \N}$ is orthogonal on moving orbits to all dynamical systems of zero entropy.
		\item [(3)] The sequence $(z(n))_{n\in \N}$ is strongly orthogonal on moving orbits to all dynamical systems of zero entropy.
		\item [(4)] The sequence $(z(n))_{n\in \N}$ is uniformly orthogonal to all dynamical systems of zero entropy.
	\end{itemize}
\end{prop}

\subsection{Chowla sequences}

Now we begin to show how to generalize Chowla property.
Let $(z(n))_{n\in \N}\in (S^1\cup \{0\})^{\mathbb{N}}$. For $m\in \N_{\ge 2}$, we denote by $$U(m)=\{e^{2\pi i \frac{k}{m}}: 0\le k\le m-1  \}$$ the set consisting of all $m$-th roots of the unity. We define supplementarily $U(\infty)=S^1$, which is considered as the set of points whose infinity norm is $1$. 
The \textit{index} of the sequence $(z(n))_{n\in \N}$ is defined by the smallest number $m\in \N_{\ge 2}\cup \{\infty \}$ such that the set 
$$\{n\in \N:z(n)\notin U(m)\cup \{0\} \} $$
has zero density in $\mathbb{N}$, that is,
$$
\lim\limits_{N\to +\infty} \frac{\sharp \{0\le n\le N-1: z(n) \not\in U(m)\cup\{0\} \}}{N}=0.
$$
We will show some examples to illustrate this definition.
\begin{exs}\label{EX: index of sequence}
	\begin{itemize}
		\item [(1)] Every sequence $(z(n))_{n\in \N}\in \{-1,0,1\}^{\mathbb{N}}$ has the index $2$.
		\item [(2)] For $\alpha\in \mathbb{R}\setminus \mathbb{Q}$, the sequence $(e^{2\pi i n\alpha})_{n\in \N}$ has the index $\infty$. In general, every sequence $(z(n))_{n\in \N}$ taking values in $S^1\setminus e(\mathbb{Q})$ has the index $\infty$ where $e(\mathbb{Q})$ denotes the set $\{e^{2\pi i a}: a\in \mathbb{Q} \}$.
		\item [(3)] For any non-integer number $\beta$, the sequence $(e^{2\pi i \beta^n})_{n\in \N}$ has the index $\infty$. In fact, if $\beta$ is an irrational number, then the sequence $(e^{2\pi i \beta^n})_{n\in \N}$ taking values in $S^1\setminus e(\mathbb{Q})$ and thus has the index $\infty$; if $\beta\in \mathbb{Q}\setminus \mathbb{Z}$, say $\beta=p/q$ with $q>1$ and ${gcd}(p,q)=1$, then $e^{2\pi i \beta^n}\in U(q^n)$ and $e^{2\pi i \beta^n}\not\in U(m)$ for any $m<q^n$ which implies that the index of the sequence $(e^{2\pi i \beta^n})_{n\in \N}$ is $\infty$.
	\end{itemize}
\end{exs}


The sequences $(S^1\cup \{0\})^{\mathbb{N}}$ are divided into two situations according to the cases when their indices are finite or infinite. It is natural that Chowla property is thus defined for these two cases separately. 

We first define the Chowla property for sequences of finite index. 
\begin{df}\label{(Ch) Sequences_finit}
	Suppose that a sequence $z\in (S^1\cup \{0\})^{\mathbb{N}}$ has the index $m\in \N_{\ge 2}$. We say that the sequence $z$ has \textbf{Chowla property} if 
	\begin{equation}\label{eq: Ch sequence_finit}
	\lim\limits_{N\to \infty}\frac{1}{N}\sum_{n=0}^{N-1} z^{i_0}(n+a_1)\cdot z^{i_1}(n+a_2)\cdot \dots \cdot z^{i_r}(n+a_r)=0,
	\end{equation}
	for each choice of $0\le a_1< \dots <a_r, r\ge 0, i_s\in\{1,2,\dots, m\}$ not all equal to $m$.
\end{df}
As it has already been noticed in Example \ref{EX: index of sequence} $(1)$,  a sequence taking values in $\{-1,0,1\}$ has the index $2$. It follows easily that a sequence taking values in $\{-1,0,1\}$ has Chowla property if and only if it satisfies the condition $\Ch$ defined in \cite{AbaKulLemDe2014}.

To define Chowla property for sequences of infinite index, we will need additionally the following. For a point $x\in S^1\cup \{0\}$ and $n\in \mathbb{Z}$, we define
$$
x^{(n)}:=
\begin{cases}
x^n &~\text{if}~x\in S^1,\\
0 &~\text{if}~x=0,
\end{cases}
$$
which is viewed as an extension of power functions on $S^1$.
It is easy to see that $(xy)^{(n)}=x^{(n)}y^{(n)}=y^{(n)}x^{(n)}$ for all $x,y\in  S^1\cup \{0\}$ and for all $n\in \Z$. Now we define the Chowla property for sequences of infinite index.

\begin{df}\label{(Ch) Sequences_infinit}
	Suppose that a sequence $z\in (S^1\cup \{0\})^{\mathbb{N}}$ has the index infinite. We say that the sequence $z$ has \textbf{Chowla property} if 
	\begin{equation}\label{eq: Ch sequence_infinit}
	\lim\limits_{N\to \infty}\frac{1}{N}\sum_{n=0}^{N-1} z^{(i_0)}(n+a_1)\cdot z^{(i_1)}(n+a_2)\cdot \dots \cdot z^{(i_r)}(n+a_r)=0,
	\end{equation}
	for each choice of $0\le a_1< \dots <a_r, r\ge 0, i_s\in\mathbb{Z}$ not all equal to $0$.
\end{df}

We revisit the functions $x\mapsto x^{(n)}$ for $n\in \mathbb{Z}$. For a fixed $m\in \N_{\ge2}$, it is easy to check that
$$
x^{(i)}=x^{(j)}~\text{for all}~x\in U(m)\cup\{0\}\Longleftrightarrow i\equiv j \mod{m}.
$$
We define 
$$
I(m):=\begin{cases}
\{0,1,\dots,m-1 \}, &~\text{if}~m\in \N_{\ge2}; \\
\mathbb{Z}, &~\text{if}~m=\infty,
\end{cases}
$$
which identifies the set of functions having the form $x\mapsto x^{(n)}$ for $n\in \Z$.
This leads us to combine Definition \ref{(Ch) Sequences_finit} and Definition \ref{(Ch) Sequences_infinit} together.



\begin{df}\label{(Ch) Sequences}
	Suppose that a sequence $z\in (S^1\cup \{0\})^{\mathbb{N}}$ has an index $m\in \N_{\ge 2}\cup \{\infty \}$. We say that the sequence $z$ has \textbf{Chowla property} if 
	\begin{equation}\label{eq: Ch sequence}
	\lim\limits_{N\to \infty}\frac{1}{N}\sum_{n=0}^{N-1} z^{(i_0)}(n+a_1)\cdot z^{(i_1)}(n+a_2)\cdot \dots \cdot z^{(i_r)}(n+a_r)=0,
	\end{equation}
	for each choice of $0\le a_1< \dots <a_r, r\ge 0, i_s\in I(m)$ not all equal to $0$.
\end{df}
It is not hard to check that Definition \ref{(Ch) Sequences} is an unification of Definition \ref{(Ch) Sequences_finit} and Definition \ref{(Ch) Sequences_infinit}. A sequence will be called a \textit{Chowla sequence} if it has Chowla property.


\section{Sequences $(e^{e\pi i \beta^n})_{n\in \N}$ for $\beta>1$}\label{Sec: examples}

In this section, we prove that the sequences $(e^{2\pi i \beta^n})_{n\in \N}$ are Chowla sequences for almost all $\beta>1$. Actually, we prove that a more general class of sequences are Chowla sequences. To this end, we study the distribution of a given sequence $(e^{2\pi i x(n)})_{n\in \N}$.


In fact, the problem about the distribution of a sequence $(e^{2\pi i x(n)})_{n\in \N}$ is actually the one about the distribution of the sequence $(\{x(n)\})_{n\in \N}$ on the unite interval $[0,1)$ where $\{x \}$ denotes the fractional part of the real number $x$. Recall that a sequence $(x(n))_{n\in \N}$ of real numbers is said to be {\em uniformly distributed modulo one} if, for every real numbers $u, v$ with $0\le u< v\le 1,$ we have
$$
\lim\limits_{N\to \infty} \frac{\sharp \{n:1\le n\le N, u\le \{x_n\}\le v  \}}{N}=v-u.
$$

For a given real number $\beta>1$, only few results are known on the distribution of the sequence $(\beta^n)_{n\in \N}$. For example, we still do not know whether $0$ is a limit point of $(\{e^n \})_{n\in \N}$, nor of $(\{(\frac{3}{2})^n\})_{n\in \N}$. However, several metric statements have been established. The first one was obtained by Koksma \cite{Kok1935}, who proved that for almost every $\beta>1$ the sequence $(\beta^n)_{n\in \N}$ is uniformly distributed modulo one on the unit interval $[0,1)$. Actually, Koksma \cite{Kok1935} proved the following theorem. 

\begin{thm}[Koksma's theorem]
	Let $(f_n)_{n\in \N}$ be a sequence of real-valued $C^1$ functions defined on an interval $[a,b]\subset \R$. Suppose that
	\begin{itemize}
		\item[(1)] $f_m'-f_n'$ is monotone on $[a,b]$ for all $m\not=n$;
		\item[(2)] $\inf_{m\not=n}\min_{x\in [a,b]} |f_m'(x)-f_n'(x)|>0$.
	\end{itemize}
	Then for almost every $x\in [a,b]$, the sequence $(f_n(x))_{n\in \N}$ is uniformly distributed modulo one on the unit interval $[0,1)$.
\end{thm}

Akiyama and Jiang \cite{AkiJia2016} proved that for almost every $\beta>1$ the sequence $(e^{2\pi i \beta^n})_{n\in \N}$ is fully oscillating. Actually, they proved the following theorem. 
\begin{thm}[\cite{AkiJia2016}, Corollary 2]\label{AJ Theorem}
	Given any real number $\alpha\not= 0$ and any $g\in C^2_+((1,\infty))$ where $C^2_+((1,+\infty))$ is the space of all positive 2-times continuously differentiable functions on $(1,+\infty)$ whose $i$-th derivative is non-negative for $i=1,2$. Then for almost every real number $\beta>1$ and for every real polynomial $P$, the sequence
	$$
	\left(e^{2\pi i (\alpha \beta^n g(\beta)+P(n))}\right)_{n\in \N}
	$$
	is fully oscillating. 
\end{thm}
The main step of their proof is to show the uniform distribution modulo one of the sequence $(\alpha \beta^n g(\beta))_{n\in \N}$ for a given $\alpha$ and a function $g\in C^2_+((1,\infty))$, for almost every real number $\beta>1$ by using Koksma's theorem.

El Abdalaoui \cite{Aba2017} also studied such class of sequences and proved that these sequences satisfied (\ref{eq:(5)}): \\

{\em For any $\alpha\not=0$, for any any $g\in C^2_+((1,\infty))$, for any real polynomials $P$ and for any increasing sequence of integers $0=b_0<b_1<b_2<\cdots$ with $b_{k+1}-b_k\to \infty$ as $k$ tends to infinity, for almost all real numbers $\beta>1$, for any $\ell\not=0$, we have} 

\begin{equation}
\lim\limits_{K\to \infty}\frac{1}{b_{K}} \sum_{k=0}^{K} \left| \sum_{n=b_k}^{b_{k+1}-1} e^{2\pi i \ell (\alpha \beta^n g(\beta)+P(n)) }  \right|=0.
\end{equation}

\vspace{1mm}

Now we prove that for almost every $\beta>1$, the sequence $(e^{2\pi i \beta^n})_{n\in \N}$ is a Chowla sequence.
\begin{thm}\label{AJexample}
	Given any $g\in C^2_{count}$ where the space $C^2_{count}$ consists of all $C^2$ functions which have at most countable zeros on $(1,+\infty)$. Then for almost all real numbers $\beta>1$, the sequences 
	\begin{equation}\label{eq: AJ sequences}
	\left(e^{2\pi i \beta^n g(\beta)} \right)_{n\in \N}
	\end{equation}
	are Chowla sequences.
\end{thm} 
\begin{proof}
	Since the function $g$ has at most countable zeros, for any $\epsilon>0$ there is a union of closed intervals $\sqcup_{i\in I} J_i^{\epsilon}\subset (1, +\infty)$ such that 
	\begin{itemize}
		\item [(1)] the index set $I$ is at most countable,
		\item  [(2)] $J_i^{\epsilon}$ is a closed interval where $g$ does not vanish,
		\item [(3)] $m((1, +\infty)\setminus\ \sqcup_{i\in I} J_i^{\epsilon} )<\epsilon$ where $m$ is the Lebesgue measure on the real line.
	\end{itemize}
	Note that if we have proved that for any $\epsilon>0$, the sequence (\ref{eq: AJ sequences}) has Chowla property for almost every $\beta\subset J_i^{\epsilon}$, then it has Chowla property for for almost every $\beta\subset (1, +\infty)$. Thus it suffices to prove that for any closed interval $J\subset (1, +\infty)$ where $g$ does not vanish, the sequence (\ref{eq: AJ sequences}) has Chowla property for almost every $\beta\subset J$.

	Let $c_\beta(n)=e^{2 \pi i \beta^{n} g(\beta)}$. For each choice of $0\le a_1< \dots <a_r, r\ge 0, i_s\in \mathbb{Z}$ not all equal to $0$, we have that 
	$$
	c_\beta^{(i_0)}(n+a_1)\cdot c_\beta^{(i_1)}(n+a_2)\cdot \dots \cdot c_\beta^{(i_r)}(n+a_r)=e^{2\pi i \beta^n h(\beta)},
	$$
	where $h(\beta)=g(\beta)\sum_{j=1}^{r} i_j \beta^{a_j}$.
	Note that $\sum_{j=1}^{r} i_j x^{a_j}$ is a non-trivial real polynomial because not all $i_s$ are equal to zero for $1\le s\le r$. It follows that the function $h$ is also in the space $C^2_{count}$.
	Thus it is sufficient to prove that for any $g\in C^2_{count}$, for any closed interval $J\subset (1, +\infty)$ where $g$ does not vanish, for almost every $\beta\in J$, and for large enough $M$, the sequence $(\beta^{n+M}g(\beta))_{n\in \N}$ is uniformly distributed modulo $1$.


	Denote by $f_n(x)=x^n g(x)$. A simple calculation allows us to see that
	\begin{equation}
	(f_m(x)-f_n(x))'=\left( (mx^{m-1}-nx^{n-1})g(x)+(x^m-x^n)g'(x)  \right),
	\end{equation}
	for all $m,n\in \mathbb{N}$.
	For given $m>n$,  we deduce that for all $x>1$, 
	\begin{equation}\label{(f_m(x)-f_n(x))'}
	\begin{split}
	\left|\frac{(f_m(x)-f_n(x))'}{x^mg(x)}\right|=&\left|\frac{m}{x}-\frac{n}{x^{m-n+1}}+(1-\frac{1}{x^{m-n}})\frac{g'(x)}{g(x)}\right|\\
	\ge & \left( \frac{m}{x}-\frac{n}{x^{m-n+1}}-(1-\frac{1}{x^{m-n}})\left|\frac{g'(x)}{g(x)}\right|  \right)\\
	\ge  & \left( \frac{m(x-1)}{x^{2}}-\left|\frac{g'(x)}{g(x)}\right| \right).
	\end{split}
	\end{equation}
	Since $J\subset (1, +\infty)$ is a closed interval, the function $|\frac{g'(x)}{g(x)}|$ is bounded and the function $\frac{(x-1)}{x^{2}}$ has a positive lower bound on $J$. Hence the right-hand side of $(\ref{(f_m(x)-f_n(x))'})$ tends to $+\infty$ when $m$ tends to $+\infty$. It follows that there is $M_1>0$ such that
	\begin{equation}
	\begin{split}
	\inf_{m\not=n\ge M_1} \min_{x\in J_i}|(f_m(x)-f_n(x))'|>0.
	\end{split}
	\end{equation}
	
	On the other hand, we have
	\begin{equation}
	\begin{split}
	(f_m'(x)-f_n'(x))'=&( (m(m-1)x^{m-2}-n(n-1)x^{n-2})g(x)\\
	&+2(mx^{m-1}-nx^{n-1})g'(x)+(x^m-x^n)g''(x)  ).
	\end{split}
	\end{equation}
	Similar to the calculation $(\ref{(f_m(x)-f_n(x))'})$, we have
	\begin{equation}\label{(f_m'(x)-f_n'(x))'}
	\begin{split}
	&\left|\frac{(f_m'(x)-f_n'(x))'}{x^mg(x)}\right|\\
	\ge  & \left( \frac{m(m-1)(x-1)}{x^3}
	-\frac{2m}{x}\left|\frac{g'(x)}{g(x)}\right|-\left|\frac{g''(x)}{g(x)}\right| \right).
	\end{split}
	\end{equation}
	Since the functions $\frac{1}{x}|\frac{g'(x)}{g(x)}|$ and $\left|\frac{g''(x)}{g(x)}\right|$ are bounded on $J$ and the function $\frac{x-1}{x^3}$ has a positive lower bound on $J$, the right-hand side of $(\ref{(f_m'(x)-f_n'(x))'})$ tends to $+\infty$ when $m$ tends to $+\infty$. It follows that there is $M_2>0$ such that
	\begin{equation}\label{|(f_m'(x)-f_n'(x))'|>0.}
	\begin{split}
	\inf_{m\not=n\ge M_2} \min_{x\in J_i}|(f_m'(x)-f_n'(x))'|>0.
	\end{split}
	\end{equation}
	This implies that $f_m'(x)-f_n'(x)$ is monotone on $J$ for $m\not=n\ge M_2$. By Koksma's theorem, for $M=\max\{M_1, M_2 \}$, the sequence $(\beta^{n+M} g(\beta))_{n\in \N} $ is uniformly distributed modulo one for almost every $\beta \in J$. This completes the proof.
	
\end{proof}


Putting $g\equiv 1$ in Theorem \ref{AJexample}, we obtain that the sequences $(e^{2\pi i \beta^n})_{n\in \N}$ are Chowla sequences for almost all $\beta>1$. By Theorem \ref{Ch sequences are S sequences} in section \ref{Sec: Ch implies S}, the sequences $(e^{2\pi i \beta^n})_{n\in \mathbb{N}}$ are orthogonal to all dynamical systems of zero entropy for almost all $\beta>1$. 
Moreover, by Proposition \ref{prop:equivalent definitions},  the sequences $(e^{2\pi i \beta^n})_{n\in \mathbb{N}}$ are strongly orthogonal on moving orbits to all dynamical systems of zero entropy for almost all $\beta>1$. On the other hand, by Corollary \ref{cor:Not SOMO positive entropy} in Section \ref{Sec: Ch implies S}, these Chowla sequences $(e^{2\pi i \beta^n})_{n\in \mathbb{N}}$ are not strongly orthogonal on moving orbits to any dynamical systems of positive entropy.

It is clear that a Chowla sequence (or even an oscillating sequence of order $1$) is aperiodic (see the definition of ``aperiodic" in Page 1, \cite{DowSer2016}). By Theorem 3.1 in \cite{DowSer2016}, we obtain that for almost every $\beta>1$, there exists a sequence of dynamical systems $(X_n, T_n)$ with entropy $h(T_n)\to \infty$ such that  the sequence $(e^{2\pi i \beta^n})_{n\in \N}$ is orthogonal to $(X_n, T_n)$ for all $n\in \N$

We remark that if $\beta$ is an algebraic number, then the sequence $(e^{2\pi i \beta^n})_{n\in \N}$ is not a Chowla sequence. In fact, since $\beta$ is an algebraic number, there exist $r$ different non-negative integers $a_1, a_2, \dots, a_r$ and $r$ integers $i_1, i_2, \dots, i_r$ not all equal to $0$ such that 
\begin{equation}\label{eq: algebraic number}
i_1\beta^{a_1}+i_2\beta^{a_2}+\dots+i_r\beta^{a_r}=0.
\end{equation}
Denote by $z(n)=e^{2\pi i \beta^n}$. Due to (\ref{eq: algebraic number}), we obtain
$$
z(n+a_1)^{i_1}z(n+a_2)^{i_2}\cdots z(n+a_2)^{i_2}=e^{2\pi i \beta^n(i_1\beta^{a_1}+i_2\beta^{a_2}+\dots+i_r\beta^{a_r})}=1.
$$
It follows that
\begin{equation}\label{eq: limit algebraic number}
\lim\limits_{N\to +\infty} \frac{1}{N}\sum_{n=0}^{N-1}z(n+a_1)^{i_1}z(n+a_2)^{i_2}\dots z(n+a_2)^{i_2}=1.
\end{equation}
But to have Chowla property requires the limit (\ref{eq: limit algebraic number}) equal to $0$. Therefore, the sequence $(e^{2\pi i \beta^n})_{n\in \N}$ is not a Chowla sequence for any algebraic number $\beta$. 



\section{Chowla property implies Sarnak property}\label{Sec: Ch implies S}


In this section, we show that Chowla property implies Sarnak property. The main theorem of this section is the following.
\begin{thm}\label{Ch sequences are S sequences}
	Chowla sequences are always orthogonal to all topological dynamical systems of zero entropy.
\end{thm}

Before we really get into the proof, let us remark that the idea of the proof of Theorem \ref{Ch sequences are S sequences} is from the proof of Theorem 4.10 in \cite{AbaKulLemDe2014}: the condition $\Ch$ implies the condition $\S$. Our main observation is the following lemma, which follows directly from the Stone-Weierstrass theorem. 
\begin{lem}\label{lem: main observation}
	The linear subspace generated by the constants and the family 
	$$\mathcal{F}:=\{F^{(i_1)}\circ S^{a_1}\cdots F^{(i_r)}\circ S^{a_r}: 0\le a_1< \dots <a_r, r\ge 0, i_s\in I(m)  \}$$
	of continuous functions is a complex algebra (closed under taking products and conjugation) of functions separating points, hence it is dense in $C((U(m)\cup \{0\})^\N)$ for every $m\in \mathbb{N}_{\ge 2}\cup \{\infty\}$.
\end{lem}
The family $\mathcal{F}$ is an analogue of the family of $\mathcal{A}$ in Lemma 4.6 of \cite{AbaKulLemDe2014}. The difference is that $\mathcal{F}$ is a complex algebra but $\mathcal{A}$ is real. 

Now we provide a detailed proof of Theorem \ref{Ch sequences are S sequences}. It is well known that for every $m\in \N_{\ge 2}\cup\{\infty \}$, the system $(U(m)^\mathbb{N}, \mathcal{B}, S, \mu_m^{\otimes\mathbb{N}})$ is a Bernoulli system for every $m\in \N$, where $\mu_m$ is the equidistributed probability measure on $U(m)$ and $\mu_m^{\otimes\mathbb{N}}$ is the product measure. Since a Bernoulli system is a Kolmogorov system, the following  is classical.
\begin{prop}\label{K system}
	For any $m\in \N_{\ge 2}\cup \{\infty\}$, the system $(U(m)^\mathbb{N}, S, \mu_m^{\otimes\mathbb{N}})$ is a Kolmogorov system.
\end{prop}

For any $m\in \N_{\ge 2}\cup \{\infty\}$ and for any $S$-invariant measure $\nu$ on $(\{0,1\}^\mathbb{N}, S)$, we define the diagram
\begin{equation}\label{Dig: first}
\xymatrix{
	(U(m)^\mathbb{N}\times \{0,1\}^\mathbb{N}, S, \mu_m^{\otimes\mathbb{N}}\otimes\nu) \ar[d]_{\rho}  \ar@/^6pc/[dd]^{P_2}   &  \\
	((U(m)\cup\{0\})^\mathbb{N}, S, \widehat{\nu})\ar[d]_{\pi} & \\	
	(\{0,1\}^\mathbb{N}, S, \nu) &
}
\end{equation}
where $\rho: (x_n,y_n)_{n\in \N}\mapsto (x_ny_n)_{n\in \N}$, $\pi: (x_n)_{n\in \N}\mapsto (x_n^{(0)})_{n\in \N}$, $P_2: (x_n,y_n)_{n\in \N}\mapsto (y_n)_{n\in \N}$ and $\widehat{\nu}=\rho_*(\mu_m^\mathbb{N}\otimes\nu)$. We study  this diagram in the following proposition.

\begin{prop}\label{FirstDiagram}
	The diagram (\ref{Dig: first}) satisfies the following properties.
	\begin{itemize}
		\item [(1)] The diagram (\ref{Dig: first}) commutes.
		\item [(2)] The factor map $\pi$ is either trivial or relatively Kolmogorov.
		\item [(3)] The measure $\widehat{\nu}$ has the form:
		$$
		\widehat{\nu}(B)=\nu(\pi(B))\prod_{i=0,b_i\not=0}^{k-1}\mu_m(b_i), 
		$$
		for any cylinder $B=[b_0, b_1, \cdots, b_{k-1}]$ in $(U(m)\cup\{0\})^\mathbb{N}$.
	\end{itemize}
\end{prop}
\begin{proof}
	To prove that the diagram (\ref{Dig: first}) commutes, it suffices to prove $P_2=\pi\circ \rho$. In fact, we have
	\begin{align*}
	\pi\circ \rho((x(n),y(n))_{n\in \N})
	=\pi ((x(n)y(n))_{n\in \N})
	=(x(n)^{(0)}y(n)^{(0)})_{n\in \N}=(y(n))_{n\in \N},
	\end{align*}
	because of the facts that $x^{(0)}=1$ for all $x\in U(m)$ and $y^{(0)}=y$ for all $y\in \{0,1\}$. This completes the proof of $(1)$.
	
	By Proposition \ref{K system}, the factor $P_2$ is relatively Kolmogorov. Since the factor map $\pi$ is an intermediate factor by (1), it is either trivial or relatively Kolmogorov. This implies $(2)$.
	
	It remains to show the form of $\widehat{\nu}$. For any cylinder $B=[b_0, b_1, \cdots, b_{k-1}]$ in $(U(m)\cup\{0\})^\mathbb{N}$, we observe that $\rho^{-1}(B)$ is the cylinder 
	$$[B_0, B_1, \cdots, B_{k-1}] \times [b_0^{(0)}, b_1^{(0)}, \cdots, b_{k-1}^{(0)}],$$
	where $B_i=\{b_i\}$ if $b_i\not=0$; and $B_i=U(m)$ if $b_i=0$. Since the support of the probability measure $\mu_m$ is $U(m)$, we obtain
	\begin{align*}
	\widehat{\nu}(B)&=\rho_*(\mu_m^{\otimes\mathbb{N}}\otimes\nu)(B)=\mu_m^{\otimes\mathbb{N}}\otimes\nu(\rho^{-1}(B))\\
	&=\nu(\pi(B))\prod_{i=0,b_i\not=0}^{k-1}\mu_m(b_i),
	\end{align*}
	which completes the proof of (3).
	
\end{proof}

The following lemma follows directly from the orthogonality of continuous group characters on the group $U(m)$.
\begin{lem}\label{Integral}
	Let $m\in \N_{\ge 2}\cup \{\infty \}$. For $k\in I(m)$, we have the integral 
	$$
	\int_{U(m)} x^{k} d\mu_m(x)=
	\begin{cases}
	1 & ~~\text{if}~~k=0,\\
	0 & ~~\text{else}.
	\end{cases}
	$$
\end{lem}

Now we do some calculations in the following two lemmas. 
\begin{lem}\label{Integral in two case}
	Let $m\in \N_{\ge 2}\cup \{\infty \}$.  Let $0\le a_1< \dots <a_r, r\ge 0$ and $i_s\in I(m)$ for $1\le s\le r$. Then the integral
	\begin{equation}\label{eq: calculation 0}
	\int_{(U(m)\cup \{0\})^\mathbb{N}} F^{(i_1)}\circ S^{a_1}\cdot F^{(i_2)}\circ S^{a_2}\cdots F^{(i_r)}\circ S^{a_r} d\widehat{\nu}
	\end{equation}
	equals $0$ if not all $i_s$ are zero, and equals to
	$$
	\int_{\{0,1\}^\mathbb{N}} F\circ S^{a_1}\cdot F\circ S^{a_2}\cdots F\circ S^{a_r} d\nu,
	$$
	if all $i_s$ are zero.
\end{lem}
\begin{proof}
	The second case follows directly from the facts that $\pi_*(\widehat{\nu})=\nu$ by Proposition \ref{FirstDiagram} and $F^{(0)}\circ S^{a_s}=F\circ S^{a_s}\circ \pi$. Now we study the first case. Fix a choice of $0\le a_1< \dots <a_r, r\ge 0, i_s\in I(m)$ not all equal to $0$.
	A straightforward computation shows that
	\begin{equation}\label{eq: calculation 1}
	\begin{split}
	&\int_{(U(m)\cup\{0\})^\mathbb{N}} F^{(i_1)}\circ S^{a_1}\cdot F^{(i_2)}\circ S^{a_2}\cdots F^{(i_r)}\circ S^{a_r} d\widehat{\nu}\\
	=&	\int_{(U(m)\cup\{0\})^\mathbb{N}} x(a_1)^{(i_1)}x(a_2)^{(i_2)}\cdots x(a_r)^{(i_r)} d\widehat{\nu}(x),
	\end{split}
	\end{equation}
	where we denote  $x=(x(n))_{n\in \N}$. We write the right-hand side of (\ref{eq: calculation 1}) a sum of two parts:
	\begin{equation}\label{eq: calculation 2}
	\int_{B} x(a_1)^{(i_1)}x(a_2)^{(i_2)}\cdots x(a_r)^{(i_r)} d\widehat{\nu}(x)
	\end{equation}
	and 
	\begin{equation}\label{eq: calculation 3}
	\int_{(U(m)\cup\{0\})^\N\setminus B} x(a_1)^{(i_1)}x(a_2)^{(i_2)}\cdots x(a_r)^{(i_r)} d\widehat{\nu}(x),
	\end{equation}
	where $B$ denotes the cylinder 	$$\left\{x\in (U(m)\cup\{0\})^\N: x(a_s)\not=0, \forall 1\le s\le r  \right\}.$$
	A simple observation shows that the integral (\ref{eq: calculation 3}) equals to $0$ because at least one of $x(a_s)$ are equal to $0$ for all $1\le s\le r$. On the other hand, due to the fact that $y^{(n)}=y^{n}$ for all $y\in U(m)$ and for all $n\in I(m)$, the integral (\ref{eq: calculation 2}) equals to
	\begin{equation}\label{eq: calculation 4}
	\begin{split}
	\int_{B} x(a_1)^{i_1}\cdots x(a_r)^{i_r} d\widehat{\nu}(x)=\nu(B) \prod_{j=1}^{r} \int_{U(m)} x(a_j)^{i_j} d\mu_m(x(a_j)).
	\end{split}	
	\end{equation}	
	By Lemma \ref{Integral}, the equation (\ref{eq: calculation 4}) equals to $0$. Therefore, we conclude that in the first case the integral (\ref{eq: calculation 0}) equals to $0$.
\end{proof}

\begin{lem}\label{E=0}
	$\mathbb{E}^{\widehat{\nu}}[F|\pi^{-1}(\mathcal{B})]=0.$
\end{lem}
\begin{proof}
	It suffices to show $\mathbb{E}^{\widehat{\nu}}[F1_{\pi^{-1}(B)}]=0$ for any cylinder $B$ in $\{0,1\}^{\mathbb{N}}$. Let $B=[b_0,b_1,\dots,b_k]$ be a cylinder in $\{0,1\}^{\mathbb{N}}$. Then $\pi^{-1}(B)=[C_0,C_1,\dots,C_k]$ where $C_i=U(m)$ if $b_i=1$; $C_i=\{0\}$ if $b_i=0$ for $0\le i\le k$.
	
	If $b_0=0$, then 
	$$
	\mathbb{E}^{\widehat{\nu}}[F1_{\pi^{-1}(B)}]=\mathbb{E}^{\widehat{\nu}}[0\cdot1_{\pi^{-1}(B)}]=0.
	$$
	If $b_0\not=0$, then $C_0=U(m)$. A simple calculation shows that
	\begin{align*}
	\mathbb{E}^{\widehat{\nu}}[F1_{\pi^{-1}(B)}]
	&=\int_{\pi^{-1}(B)} x(0) d\widehat{\nu}(x)
	=\nu(B)\int_{U(m)} x(0) d\mu_m(x(0))=0.
	\end{align*}
	The last equality is due to Lemma \ref{Integral}. This completes the proof.
\end{proof}

Now we prove some criteria of Chowla property.
\begin{prop}\label{Equvenlent Of Ch}
	Let $z\in (S^1\cup \{0\})^{\mathbb{N}}$. Then the followings are equivalent.
	\begin{itemize}
		\item [(1)] The sequence $z$ has Chowla property.
		\item [(2)] Q-gen$(z)=\{\widehat{\nu}: \nu\in \text{Q-gen}(\pi(z)) \}$.
		\item[(3)] As $k\to \infty$, $\delta_{N_k, \pi(z)} \to \nu$ if and only if $\delta_{N_k, z} \to \widehat{\nu}$.
	\end{itemize}
\end{prop}

Actually, Proposition \ref{Equvenlent Of Ch} follows directly from the following lemma.

\begin{lem}\label{Week convegence Lemma}
	Let $z\in (S^1\cup \{0\})^{\mathbb{N}}$. Suppose that the sequence $z$ has an index $m\in \N_{\ge 2}\cup \{\infty \}$.  Suppose that $\pi(z)$ is quai-generic for $\nu$ along $(N_k)_{k\in \N}$.
	Then 
	$$
	\delta_{N_k, z} \to \widehat{\nu}, ~~\text{as}~~k\to \infty
	$$
	if and only if 
	\begin{equation}\label{eq: week convegence 0}
	\lim\limits_{k\to \infty}\frac{1}{N_k}\sum_{n=0}^{N_k-1} z^{(i_1)}(n+a_1)\cdot z^{(i_2)}(n+a_2)\cdot \dots \cdot z^{(i_r)}(n+a_r)=0,
	\end{equation}
	for each choice of $0\le a_1< \dots <a_r, r\ge 0, i_s\in I(m)$ not all equal to $0$. 
\end{lem}
\begin{proof}
	$``\Rightarrow"$ Fix a choice of $0\le a_1< \dots <a_r, r\ge 0, i_s\in I(m)$ not all equal to $0$. Observe that 
	\begin{equation}\label{eq: week convegence 1}
	\begin{split}
	&\frac{1}{N_k}\sum_{n=0}^{N_k-1} z^{(i_1)}(n+a_1)\cdot z^{(i_2)}(n+a_2)\cdot \dots \cdot z^{(i_r)}(n+a_r)\\
	=&\frac{1}{N_k}\sum_{n=0}^{N_k-1}   (F^{(i_1)}\circ S^{a_1}\cdot F^{(i_2)}\circ S^{a_2}\cdot \dots \cdot F^{(i_r)}\circ S^{a_r})\circ S^n(z),
	\end{split}	
	\end{equation}	
	for all $k\in \N$.
	Since the measure $\widehat{\nu}$ is quasi-generic for $z$ along $(N_k)_{k\in \N}$, we obtain
	\begin{equation}\label{eq: week convegence 2}
	\begin{split}
	&\lim\limits_{k\to \infty} \frac{1}{N_k}\sum_{n=0}^{N_k-1}   (F^{(i_1)}\circ S^{a_1}\cdot F^{(i_2)}\circ S^{a_2}\cdot \dots \cdot F^{(i_r)}\circ S^{a_r})\circ S^n(z)\\
	&=\int_{(U(m)\cup\{0\})^\mathbb{N}} F^{(i_1)}\circ S^{a_1}\cdot F^{(i_2)}\circ S^{a_2}\cdot \dots \cdot F^{(i_r)}\circ S^{a_r} d\widehat{\nu}.
	\end{split}	
	\end{equation}
	According to Lemma \ref{Integral in two case}, the right-hand side of the above equation equals to $0$, which completes the proof.
	
	$``\Leftarrow"$ Without loss of generality, we may assume that 
	$$
	\delta_{N_k, z} \to \theta, ~~\text{as}~~k\to \infty,
	$$
	for some $S$-invariant probability measure $\theta$ on $(U(m)\cup \{0\})^\mathbb{N}$.
	Now we prove that the measure $\theta$ is exactly the measure $\widehat{\nu}$.  Similarly to (\ref{eq: week convegence 1}) and (\ref{eq: week convegence 2}), we have
	\begin{align*}
	\lim\limits_{k\to \infty}&\frac{1}{N_k}\sum_{n=0}^{N_k-1} z^{(i_1)}(n+a_1)\cdot z^{(i_2)}(n+a_2)\cdot \dots \cdot z^{(i_r)}(n+a_r)\\
	&=\int_{(U(m)\cup\{0\})^\mathbb{N}} F^{(i_1)}\circ S^{a_1}\cdot F^{(i_2)}\circ S^{a_2}\cdots F^{(i_r)}\circ S^{a_r} d\theta,
	\end{align*}
	for each choice  of $0\le a_1< \dots <a_r, r\ge 0, i_s\in I(m)$. According to the hypothesis (\ref{eq: week convegence 0}), we have
	\begin{equation*}
	\int_{(U(m)\cup\{0\})^\mathbb{N}} F^{(i_1)}\circ S^{a_1}\cdot F^{(i_2)}\circ S^{a_2}\cdots F^{(i_r)}\circ S^{a_r} d\theta=0,
	\end{equation*}
	for each choice of $0\le a_1< \dots <a_r, r\ge 0, i_s\in I(m)$ not all equal to $0$. Repeating the same argument in Lemma \ref{Integral in two case}, we have
	\begin{align*}
	&\int_{(U(m)\cup\{0\})^\mathbb{N}} F^{(0)}\circ S^{a_1}\cdot F^{(0)}\circ S^{a_2}\cdots F^{(0)}\circ S^{a_r} d\theta\\
	&=\int_{\{0,1\}^\mathbb{N}} F\circ S^{a_1}\cdot F\circ S^{a_2}\cdots F\circ S^{a_r} d\nu.
	\end{align*}
	Hence
	$$
	\int_{(U(m)\cup\{0\})^\mathbb{N}} f d\theta = \int_{(U(m)\cup\{0\})^\mathbb{N}} f d\widehat{\nu},
	$$
	for all $f$ belonging to the family $\mathcal{F}$ defined in Lemma \ref{lem: main observation}.
	Since the family $\mathcal{F}$ is dense in $C(X)$, by Lemma \ref{lem: main observation}, we conclude that $\theta=\widehat{\nu}$.
\end{proof}

If additionally we assume that the sequence does not take the value $0$ in Proposition \ref{Equvenlent Of Ch}, then we obtain the following criterion of Chowla property for sequences in $(S^1)^\N$.
\begin{cor}\label{Ch not 0}
	Let $z\in (S^1)^{\mathbb{N}}$. Suppose that $z$ has an index $m\in \N_{\ge 2}\cup\{\infty \}$. Then the sequence $z$ has Chowla property if and only if $z$ is generic for $\mu_m^{\otimes\mathbb{N}}$.
\end{cor}
\begin{proof}
	Observe that $\pi(z)=(1,1,\dots)$. It follows that Q-gen$(\pi(z))=\{\delta_{(1,1,\dots)} \}$, that is, $\pi(z)$ is generic for the measure $\delta_{(1,1,\dots)}$. A straightforward computation shows that the measure $\widehat{\delta}_{(1,1,\dots)}$ is exactly $\mu_m^{\otimes\mathbb{N}}$. Due to the equivalence between $(1)$ and $(2)$ in Proposition \ref{Equvenlent Of Ch}, we conclude that the sequence $z$ has Chowla property if and only if $z$ is generic for $\mu_m^{\otimes\mathbb{N}}$.
\end{proof}

As an analogue of Corollary 15 in \cite{AbaLemDe2017}, we have the following result. The proof is similar to the proof of Corollary 15 in \cite{AbaLemDe2017}, which is omitted here.
\begin{cor}\label{cor:Not SOMO positive entropy}
	Let $z\in (S^1)^{\mathbb{N}}$. Suppose that $z$ has Chowla property. Then the sequence $z$ is not strongly orthogonal on moving orbits to any dynamical systems of positive entropy.
\end{cor}

Now we construct a new diagram. Let $(X,T)$ be a topological dynamical system and $x\in X$ be a completely deterministic point. Let $\nu$ be a $S$-invariant measure on $\{0,1\}^\mathbb{N}$. Let $z\in (U(m)\cup\{0\})^\mathbb{N}$ be quasi-generic for $\widehat{\nu}$ along $(N_k)_{k\in \N}$.
Suppose that $(x, z)$ is quasi-generic for some $T\times S$-invariant measure $\rho$ on $X\times U(m)^\mathbb{N}$ along $(N_k)_{k\in \N}$.
We define the diagram
\begin{equation}\label{Dig: Second}
\xymatrix @=0.5cm{
	&	(X\times (U(m)\cup\{0\})^\mathbb{N}, T\times S, \rho ) \ar[d]_{\gamma}  &\ar[r]^{P_2} && ((U(m)\cup\{0\})^\mathbb{N}, S, \widehat{\nu}) \ar[d]_{\pi} \\
	& (X,T,(P_{1})_{*}(\rho))\bigvee (\{0,1\}^\mathbb{N}, S, \nu)  &\ar[r]_{\sigma} &&
	(\{0,1\}^\mathbb{N}, S, \nu, \mathcal{B} )
}
\end{equation}
where the factor map $\pi$ is already defined in diagram (\ref{Dig: first}), $P_i$ is the projection on the $i$-th coordinate for $i=1,2$ and the factor maps $\gamma$ and $\sigma$ are induced by the intermediate factor $(X,T,(P_{1})_{*}(\rho))\bigvee (\{0,1\}^\mathbb{N}, S, \nu)$. 

We have the following lemma.
\begin{lem}\label{Zero Times K}
	The factor maps $\sigma$ and $\pi$ in digram (\ref{Dig: Second}) are relatively independent.	
\end{lem}
\begin{proof}
	Since $P_1$ is the projection on $X$ and the point $x$ is completely deterministic in $X$, we have the entropy $h(T, (P_{1})_{*}(\rho))=0$. Due to Pinsker's formula (c.f. \cite{p}, Theorem 6.3), the factor map $\sigma$ is of relatively entropy zero. On the other hand, by Proposition \ref{FirstDiagram}, the factor map $\pi$ is relatively Kolmogorov. According to Lemma \ref{Relatively Independent}, we conclude that the factor maps $\sigma$ and $\pi$ are relatively independent.	
\end{proof}

Now we prove the main theorem in this section by using the diagram (\ref{Dig: Second}).
\begin{proof}[Proof of Theorem \ref{Ch sequences are S sequences}]
	Let $(X,T)$ be a topological dynamical system. Let $x\in X$ be a completely deterministic point. 
	Suppose that $(x, z)$ is quasi-generic for some measure $\rho$ along $(N_k)_{k\in \N}$. For any $f\in C(X)$, since $z(n)=F(S^nz)$, we have 
	\begin{align}\label{eq: Ch are S 1}
	\lim\limits_{k\to \infty}\frac{1}{N_k}\sum_{n=0}^{N_k-1} f(T^nx)z(n)=\mathbb{E}^\rho [f\otimes F].
	\end{align}
	On the other hand, by the criteria of Chowla property in Proposition \ref{Equvenlent Of Ch}, without loss of generality, we may assume that $z$ is quasi-generic for some measure $\widehat{\nu}$ along $(N_k)_{k\in \N}$ where $\nu$ is a $S$-invariant probability measure on $\{0,1\}^\N$. Observing $(P_{2})_{*}(\rho)=\widehat{\nu}$, we establish the diagram (\ref{Dig: Second}). By Lemma \ref{Zero Times K}, the factor maps $\sigma$ and $\pi$ are relatively independent. Let $P:=(\sigma\circ\gamma, \pi\circ P_2)$. A simple computation shows that
	$$
	\mathbb{E}^\rho [f\otimes F|P^{-1}(\mathcal{B}\otimes\mathcal{B}) ]=\mathbb{E}^{\gamma_{*}(\rho)} [f| \sigma^{-1}(\mathcal{B})]\mathbb{E}^{\widehat{\nu
	}} [F| \pi^{-1}(\mathcal{B})].
	$$
	Since Lemma \ref{E=0} tells us that $\mathbb{E}^{\widehat{\nu
	}} [F| \pi^{-1}(\mathcal{B})]=0$, we obtain
	$$\mathbb{E}^\rho [f\otimes F|P^{-1}(\mathcal{B}\otimes\mathcal{B}) ]=0.$$
	Hence we have
	\begin{equation}\label{eq: Ch are S 2}
	\mathbb{E}^\rho[f\otimes F]=\mathbb{E}^\rho[\mathbb{E}^\rho [f\otimes F|P^{-1}(\mathcal{B}\otimes\mathcal{B}) ]]=0.
	\end{equation}	
	Combining (\ref{eq: Ch are S 1}) and (\ref{eq: Ch are S 2}), we complete the proof.
\end{proof}

\section{Chowla property for random sequences}\label{Sec: Ch random variable}
In this section, we study sequences of random variables having Chowla property almost surely. We give a sufficient and necessary condition for a stationary random sequence to be a Chowla sequence almost surely. We also construct some dependent random sequences having Chowla property almost surely.



\subsection{Chowla property for stationary random sequences} ~~~\hspace{1cm}\\
Here, we study sequences of stationary random variables.
\begin{prop}\label{Prop: StationaryDynamical}
	Let $\mathcal{Z}=(Z_n)_{n\in \N}$ be a sequence of stationary random variables taking values in $S^1\cup \{0\}$. Suppose that the random sequence $\mathcal{Z}$ has an index $m\in \N_{\ge 2}\cup\{\infty \}$ almost surely. Then almost surely the random sequence $\mathcal{Z}$ is a Chowla sequence if and only if 
	\begin{equation}\label{eq: StationaryDynamical}
	\mathbb{E}[Z_{a_1}^{(i_1)}\cdot Z_{a_2}^{(i_2)}\cdot \dots \cdot Z_{a_r}^{(i_r)}]=0,
	\end{equation}
	for each choice of $1\le a_1< \dots < a_r, r\ge 0, i_s\in I(m)$ not all equal to $0$.  	
\end{prop}
\begin{proof}
	$``\Rightarrow"$ Since the random sequence $\mathcal{Z}$ is stationary, it can be viewed as a sequence $(f(T^n))_{n\in \N}$ for some $\mathcal{C}$-measurable complex function $f$ and $\mathcal{C}$-measurable transformation $T:X\to X$ where $X$ is a compact metric space and $\mathcal{C}$ is a $\sigma$-algebra on $X$. By Birkhoff's ergodic Theory, for any $1\le a_1< \dots < a_r, r\ge 0, i_s\in I(m)$ not all equal to $0$, almost surely we have 
	\begin{align*}
	&\mathbb{E}[Z_{a_1}^{(i_1)}\cdot Z_{a_2}^{(i_2)}\cdot \dots \cdot Z_{a_r}^{(i_r)}|\mathcal{I}]\\
	=&\lim\limits_{N\to \infty}\frac{1}{N}\sum_{n\le N} Z_{n+a_1}^{(i_1)}\cdot Z_{n+a_2}^{(i_2)}\cdot \dots \cdot Z_{n+a_r}^{(i_r)}=0,
	\end{align*}
	where $\mathbb{E}[f|{\mathcal{I}}]$  is the expectation of $f$ conditional on the $\sigma$-algebra $\mathcal {I}$ consisting of invariant sets of $T$. Thus we conclude that
	$$
	\mathbb{E}[Z_{a_1}^{(i_1)}\cdot Z_{a_2}^{(i_2)}\cdot \dots \cdot Z_{a_r}^{i_r}]=\mathbb{E}[\mathbb{E}[Z_{a_1}^{(i_1)}\cdot Z_{a_2}^{(i_2)}\cdot \dots \cdot Z_{a_r}^{(i_r)}|\mathcal{I}]]=0,
	$$
	for any $1\le a_1< \dots < a_r, r\ge 0, i_s\in I(m)$ not all equal to $0$.
	
	$``\Leftarrow"$ Fix a choice of $1\le a_1< \dots < a_r, r\ge 0, i_s\in I(m)$ not all equal to $0$. Let $s=\min\{1\le j\le r: i_j\not=0\}$. Define the random variables 
	$$
	Y_n:=Z_{a_1+n}^{(i_1)}\cdot Z_{a_2+n}^{(i_2)}\cdot \dots \cdot Z_{a_r+n}^{(i_r)}, ~~\text{for all}~~n\in \mathbb{N}.
	$$
	Then for $k> 1$ and $n\ge 1$, the product $Y_n\cdot Y_{n+k}$ has the form
	\begin{equation}\label{eq: Stationary dynamical}
	Z_{a_1+n}^{(i_1')}\cdot Z_{a_2+n}^{(i_2')}\cdot \dots \cdot Z_{a_r+n+k}^{(i_{r+k-1}')}
	\end{equation}
	where $i_j'\in I(m)$ and $i_j'=i_j+i_{j-k+1} (\text{mod}~~m)$ \footnote{In particular, if $m=\infty$, then $i_j'=i_j+i_{j-k+1}$ for $1\le j\le r+k-1$.} for $1\le j\le r+k-1$ with $i_l:=0$ for $l\le 0$ or $l\ge r+1$.   It is not hard to see that $i_s'\not=0$. It follows that the expectation of (\ref{eq: Stationary dynamical}) vanishes by the hypothesis (\ref{eq: StationaryDynamical}), that is,
	$$\mathbb{E}[Y_nY_{n+k}]=0, ~\text{for all}~ n,k\ge 1.$$
	This means that the random sequence $(Y_n)_{n\in \N}$ is orthogonal \footnote{Recall that a random sequence $(Y_n)_{n\in \N}$ is \textit{orthogonal}, if $\mathbb{E}[Y_nY_m]=0$ for any $n\not=m$.}. Therefore almost surely we have 
	$$
	\lim\limits_{N\to \infty} \frac{1}{N}\sum_{n=0}^{N-1} Y_n=0,
	$$
	i.e.
	$$
	\lim\limits_{N\to \infty}\frac{1}{N}\sum_{n\le N} Z_{n+a_1}^{(i_1)}\cdot Z_{n+a_2}^{(i_2)}\cdot \dots \cdot Z_{n+a_r}^{(i_r)}=0.
	$$
	Since there are only countablely many choices of $1\le a_1< \dots < a_r, r\ge 0, i_s\in I(m)$ not all equal to $0$, we deduce that almost surely
	$$
	\lim\limits_{N\to \infty}\frac{1}{N}\sum_{n\le N} Z_{n+a_1}^{(i_1)}\cdot Z_{n+a_2}^{(i_2)}\cdot \dots \cdot Z_{n+a_r}^{(i_r)}=0,
	$$
	for all choices of $1\le a_1< \dots < a_r, r\ge 0, i_s\in I(m)$ not all equal to $0$.
\end{proof}

We remark that if we drop the assumption that the random sequence $(Z_n)_{n\in \N}$ is stationary, then the proof of the sufficiency in Proposition \ref{Prop: StationaryDynamical} is still valid.

If additionally we suppose the random variables are independent and identically distributed, then we have the following equivalent condition of Chowla property for random sequences.
\begin{cor}
	Suppose that $\mathcal{Z}=(Z_n)_{n\in \N}$ is a sequence of independent and identically distributed random variables taking values in $S^1\cup \{0\}$. Let $m\in \N_{\ge 2}\cup \{0\}$. Suppose that $\mathcal{Z}$ has the index $m$ almost surely. Then $\mathcal{Z}$ is a Chowla sequence almost surely if and only if $\mathbb{E}[Z_1^{(i)}]=0$ for all $i\in I(m)\setminus \{0\}$.
\end{cor}
\begin{proof}
	Since $\mathcal{Z}$ is independent and identically distributed, we observe that the equivalent condition $(\ref{eq: StationaryDynamical})$ of Chowla property in Proposition \ref{Prop: StationaryDynamical} becomes 
	$$
	\mathbb{E}[Z_{a_1}^{(i_1)}]\cdot\mathbb{E}[Z_{a_2}^{(i_2)}]\cdot\cdots \cdot \mathbb{E}[Z_{a_r}^{(i_r)}]=0,
	$$
	for each choice of $1\le a_1< \dots < a_r, r\ge 0, i_s\in I(m)$ not all equal to $0$.  This is equivalent to say that $\mathbb{E}[Z_1^{(i)}]=0$ for all $i\in I(m)\setminus \{0\}$.
\end{proof}

We show an example of sequences of independent and identically distributed random variables having Chowla property almost surely.
\begin{ex}
	For any $m\in \N_{\ge 2}\cup\{\infty \}$, let $((U(m)\cup\{0\})^{\mathbb{N}}, S, \mu_m^{\otimes\mathbb{N}})$ be a measurable dynamical system. Recall that $F: (U(m)\cup\{0\})^{\mathbb{N}}\to (U(m)\cup\{0\})$ is the projection on the first coordinate. Let $Z_n:=F\circ S^n$. Then $(Z_n)_{n\in \mathbb{N}}$ is a sequence of independent and identically distributed random variables, which has the index $m$ almost surely. It is easy to check that $\mathbb{E}[Z_1^{(i)}]=0$ for all $i\in I(m)\setminus \{0\}$ by Lemma \ref{Integral}. Therefore, almost surely the sequence $(Z_n)_{n\in \N}$ is a Chowla sequence.
\end{ex}

A natural question is raised whether there are sequences of dependent random variables having Chowla property almost surely. We will prove later (in Section \ref{Subsec: dependent}) that such sequences exist. For this propose, we prove additionally the following lemma.

\begin{lem}\label{lem: moment implies indepence}
	Let $X_1, X_2, \dots, X_k$ be bounded random variables. The random variables $X_1, \dots, X_k$ are independent if and only if
	\begin{equation}\label{eq: moment implies indepence 0}
	\mathbb{E}[X_1^{n_1}X_2^{n_2}\dots X_k^{n_k}]=\mathbb{E}[X_1^{n_1}]\mathbb{E}[X_2^{n_2}]\dots \mathbb{E}[X_k^{n_k}]
	\end{equation}
	for all positive integers $n_1,\dots,n_k$. 
\end{lem}

\begin{proof}
	The necessity being trivial, we just prove the sufficiency. A simple computation allows us to see that if
	$$\mathbb{E}[f_{1,i}(X_1)f_{2,i}(X_2)\dots f_{k,i}(X_k)]=\mathbb{E}[f_{1,i}(X_1)]\mathbb{E}[f_{2,i}(X_2)]\dots \mathbb{E}[f_{k,i}(X_k)]
	$$
	holds for functions $f_{j,i}$ for all $1\le j\le k$ and for all $1\le i\le l$,
	then 
	\begin{equation}\label{eq: moment implies indepence 1}
	\begin{split}
	&\mathbb{E}[\sum_{i=1}^{l}f_{1,i}(X_1)\sum_{i=1}^{l}f_{2,i}(X_2)\dots \sum_{i=1}^{l}f_{k,i}(X_k)]\\
	&=\mathbb{E}[\sum_{i=1}^{l}f_{1,i}(X_1)]\mathbb{E}[\sum_{i=1}^{l}f_{2,i}(X_2)]\dots \mathbb{E}[\sum_{i=1}^{l}f_{k,i}(X_k)].
	\end{split}
	\end{equation}	
	Combining (\ref{eq: moment implies indepence 0}) and (\ref{eq: moment implies indepence 1}), we have
	\begin{equation}\label{eq: moment implies indepence 2}
	\mathbb{E}[P_1(X_1)\dots P_k(X_k)]=\mathbb{E}[P_1(X_1)]\mathbb{E}\dots \mathbb{E}[P_k(X_k)],
	\end{equation}
	for all $P_1, \dots, P_k$ complex polynomials.
	It is well known that the random variables $X_1, \dots, X_k$ are independent if and only if for any continuous functions $f_1, f_2, \dots, f_k$,
	\begin{equation}\label{eq: moment implies indepence 3}
	\mathbb{E}[f_1(X_1)\dots f_k(X_k)]=\mathbb{E}[f_1(X_1)]\dots \mathbb{E}[f_k(X_k)].
	\end{equation}
	It is noticed that if additionally the random variable $X_1, \dots, X_k$ are bounded, then  ``continuous functions'' can be replaced by ``bounded continuous functions'' in the above statement. 
	It is classical that any continuous bounded function is approximated by polynomials under uniform norm. Due to (\ref{eq: moment implies indepence 2}) and the approximation by polynomials, we see that (\ref{eq: moment implies indepence 3}) holds for any bounded continuous functions  $f_1, f_2, \dots, f_k$,  which implies that the random variables $X_1, \dots, X_k$ are independent.
\end{proof}

Now we show a sufficient and necessary condition for a sequence of stationary random variables having Chowla property almost surely to be independent.
\begin{cor}\label{Cor: Ch then independent}
	Let $(X_n)_{n\in \N}$ be a sequence of stationary random variables taking values in $S^1\cup \{0\}$. Let $m\in \N_{\ge 2}\cup \{0\}$. Suppose that almost surely the sequence $(X_n)_{n\in \N}$ has the index $m$ and has Chowla property. Then 
	\begin{itemize}
		\item [(1)] if $m$ is finite, then the sequence $(X_n)_{n\in \N}$ is independent if and only if 
		$$
		\mathbb{E}[|X_{a_1}X_{a_2}\cdots X_{a_r}|]=\mathbb{E}[|X_{a_1}|]\mathbb{E}[|X_{a_2}|]\cdots \mathbb{E}[|X_{a_r}|]
		$$
		for each choice of $1\le a_1< \dots < a_r, r\ge 0;$
		\item [(2)] if $m$ is infinite, then the sequence $(X_n)_{n\in \N}$ is always independent.
	\end{itemize}
	
\end{cor}
\begin{proof}
	We first consider the case when the index $m$ is finite. Since the sequence $(X_n)_{n\in \N}$ has Chowla property almost surely, by Proposition \ref{Prop: StationaryDynamical}, we obtain that for each choice of $1\le a_1< \dots < a_r, r\ge 0, i_s\in \mathbb{N}\setminus\{0\}$ not all equal to the multiples of $m$,  
	\begin{equation}
	\mathbb{E}[Z_{a_1}^{i_1}Z_{a_2}^{i_2}\cdots Z_{a_r}^{i_r}]=0,
	\end{equation} 
	and at least one of $
	\mathbb{E}[Z_{a_1}^{i_1}], \mathbb{E}[Z_{a_2}^{i_2}], \cdots , \mathbb{E}[Z_{a_r}^{i_r}]$ equal to zero. It follows that
	\begin{equation}\label{eq: Ch then independent 1}
	\mathbb{E}[Z_{a_1}^{i_1}]\cdot\mathbb{E}[Z_{a_2}^{i_2}]\cdot\cdots \cdot \mathbb{E}[Z_{a_r}^{i_r}]=0=\mathbb{E}[Z_{a_1}^{i_1}Z_{a_2}^{i_2}\cdots Z_{a_r}^{i_r}],
	\end{equation}
	for each choice of $1\le a_1< \dots < a_r, r\ge 0, i_s\in \mathbb{N}\setminus\{0\}$ not all equal to the multiples of $m$. Due to Lemma \ref{lem: moment implies indepence} and the equation (\ref{eq: Ch then independent 1}), the sequence $(X_n)_{n\in \N}$ is independent if and only if 
	\begin{equation}
	\mathbb{E}[Z_{a_1}^{k_1m}]\cdot\mathbb{E}[Z_{a_2}^{k_2m}]\cdot\cdots \cdot \mathbb{E}[Z_{a_r}^{k_rm}]=\mathbb{E}[Z_{a_1}^{k_1m}] \mathbb{E}[Z_{a_2}^{k_2m}]\cdots \mathbb{E}[Z_{a_r}^{k_rm}],
	\end{equation}
	for each choice of $1\le a_1< \dots < a_r, r\ge 0, k_s\in \mathbb{N}\setminus\{0\}$. 
	It is not hard to see that $X_n^{km}=|X_n|$ for every $n\in \N$ and for every positive integer $k$. It follows that the sequence $(X_n)_{n\in \N}$ is independent if and only if 
	$$
	\mathbb{E}[|X_{a_1}X_{a_2}\cdots X_{a_r}|]=\mathbb{E}[|X_{a_1}|]\mathbb{E}[|X_{a_2}|]\cdots \mathbb{E}[|X_{a_r}|],
	$$
	for each choice of $1\le a_1< \dots < a_r, r\ge 0.$ This completes the proof of (1).
	
	Now we consider the case when the index $m$ is infinite. Repeating the same reasoning with (\ref{eq: Ch then independent 1}), we obtain 
	\begin{equation}\label{eq: Ch then independent 2}
	\mathbb{E}[Z_{a_1}^{i_1}]\cdot\mathbb{E}[Z_{a_2}^{i_2}]\cdot\cdots \cdot \mathbb{E}[Z_{a_r}^{i_r}]=\mathbb{E}[Z_{a_1}^{i_1}Z_{a_2}^{i_2}\cdots Z_{a_r}^{i_r}]=0,
	\end{equation}
	for each choice of $1\le a_1< \dots < a_r, r\ge 0, i_s\in \mathbb{N}\setminus\{0\}$.  According to Lemma \ref{lem: moment implies indepence} and the equation (\ref{eq: Ch then independent 2}), the sequence $(X_n)_{n\in \N}$ is independent.
\end{proof}

\subsection{Chowla property for independent random sequences}

\hspace{1cm}\\
Now we study Chowla property for independent random sequences. A sequence taking values in $\cup_{m\in \N_{\ge2}}U(m)$, having index $\infty$ and having Chowla property almost surely is also constructed in this subsection.

As it has already been  noticed just after Proposition \ref{Prop: StationaryDynamical}, the proof of the sufficiency in Proposition \ref{Prop: StationaryDynamical} is still valid without the assumption of stationarity. If additionally assuming that the random sequence is independent, we obtain the following proposition which is a direct consequence of the sufficiency in Proposition \ref{Prop: StationaryDynamical}.

\begin{prop}\label{Prop: Independent implies Ch}
	Let $(X_n)_{n\in \N}$ be a sequence of independent random variables which take values in $S^1\cup \{0\}$. Let $m\in \N_{\ge 2}\cup \{0\}$. Suppose that almost surely the sequence $(X_n)_{n\in \N}$ has the index $m$. Suppose that for any $k\in I(m)\setminus\{0\}$, there exists $N>0$ such that for any $n>N$, the expectation $\mathbb{E}[X_n^{k}]=0$. Then almost surely the sequence $(X_n)_{n\in \N}$ is a Chowla sequence.
\end{prop}

We remark that if the index $m$ is finite, then the value $N$ in the sufficient condition in Proposition \ref{Prop: Independent implies Ch} is  uniform for all $k\in I(m)\setminus \{0\}$.

We show the existence of sequences taking values in $\cup_{m\in \N_{\ge2}}U(m)$, having index $\infty$ and having Chowla property almost surely in the following example.

\begin{ex}
	
	Let $(Z_n)_{n\ge 2}$ be a sequence of independent random variables on  $(S^1\cup\{0\})^{\mathbb{N}}$ such that $Z_n$ has the distribution $\mu_n$ for $n\ge 2$.
	It easily follows that almost surely the index of $(Z_n)_{n\ge 2}$ is infinite. 
	It is not difficult to see that 
	$$
	\mathbb{E}[Z_n^k]=0,~\text{for all}~n\ge 2 ~\text{and for all}~k\in I(n)\setminus\{0\}.
	$$
	Thus the sequence $(Z_n)_{n\ge 2}$ satisfies the condition in Proposition \ref{Prop: Independent implies Ch}. Therefore, almost surely the sequence $(Z_n)_{n\ge 2}$ is a Chowla sequence.
\end{ex} 

\subsection{Chowla property for dependent random sequences}\label{Subsec: dependent}
In this subsection, we construct sequences of dependent random variables which have Chowla property almost surely. As it has been proved in Corollary \ref{Cor: Ch then independent} that the sequences of stationary random variables having the index $\infty$ and having Chowla property almost surely are always independent, we will concentrate on the situation where the index is a finite number. 
In what follows, we construct sequences of dependent stationary random variables which have the index $2$ and have Chowla property almost surely. The sequences of dependent stationary random variables which have the finite index $m>2$ and have Chowla property almost surely can be constructed in a similar way.

Let $\mathcal{A}$ be a finite alphabet. Now we study the symbolic dynamics $(\mathcal{A}^{\mathbb{N}}, S, \mu)$ where $\mu$ is the uniform-Bernoulli measure. Let $G:\mathcal{A}^{\mathbb{N}}\to \{-1,0,1\}$ be a continuous function. Denote by $Z_{G,n}:=G\circ S^n$. We are concerned with the random sequence $(Z_{G,n})_{n\in \N}$. It is easy to check that the sequence $(Z_{G,n})_{n\in \N}$ has the index $2$ almost surely.

Obviously, $G^{-1}(t)$ is a clopen subset of $\mathcal{A}^{\mathbb{N}}$ for each $t\in \{-1,0,1\}$. It follows that the function $G$ is decided by finite words, i.e., there exist $l\in \mathbb{N}\setminus\{0\}$ and a function $g$ on $\mathcal{A}^l$ such that for all $x=(x_n)_{n\in \N}\in \mathcal{A}^{\mathbb{N}}$,
$$
G(x)=g(x_0, x_1, \dots, x_{l-1}).
$$

We associate a rooted 3-color tree $\mathcal{T}_G$ with the sequence $(Z_{G,n})_{n\in \N}$.

\vspace{6pt}
{\em We first draw a root as the $0$-th level. We then draw $\sharp \mathcal{A}$ descendants from the vertex in $0$-th level as the $1$-th level. The vertices at the first level are entitled by the elements in $\mathcal{A}$. Draw black the edges between the vertices of the $0$-th level and the vertices of the $1$-th level. Repeat this until we get the $(l-1)$-th level and thus obtain a finite tree with only black edges and the property that each vertex has $\sharp \mathcal{A}$ descendants. Observe that the branch from the root to a vertex in the $(l-1)$-th level can be viewed as a word of length $(l-1)$. Now we show how to draw the $l$-th level. A vertex identified by $a_1a_2\dots a_{l-1}$ at $(l-1)$-th level has a descendant entitled $a\in \mathcal{A}$ if $g(a_1,a_2,\dots,a_{l-1}, a)\not=0$. Moreover, the edge between them is red if $g(a_1,a_2,\dots,a_{l-1}, a)=1$ and is green if $g(a_1,a_2,\dots,a_{l-1}, a)=-1$. In general, for $k\ge l-1$, a vertex identified by $a_1a_2\dots a_{k}$ in $k$-th level has a descendant entitled $a\in \mathcal{A}$ if and only if $g(a_{k-l+2},a_{k-l+3},\dots,a_{k}, a)\not=0$ and the color of the edge between them is red if  $g(a_{k-l+2},a_{k-l+3},\dots,a_{k}, a)=1$ and green if it has the value $-1$. There edges are called the edges of the $k$-th level. Therefore, the rooted 3-color tree that we obtained is named by $\mathcal{T}_G$.}

\vspace{6pt}

We illustrate this by the following examples.
\begin{ex}\label{example}
	\begin{itemize}
		\item [(1)] Let $\mathcal{A}=\{0,1,2\}$. Define 
		$g_1: \mathcal{A}^2\to \{-1,0,1\}$ by 
		\begin{align*}
		&(0, 0)\mapsto 1,  (1, 2)\mapsto 1, (2, 1)\mapsto 1,\\
		&(0,1)\mapsto -1, (1, 0)\mapsto -1,  (2, 2)\mapsto -1,\\
		&(2,0)\mapsto 0, (0,2) \mapsto 0, (1,1) \mapsto 0.
		\end{align*}
		Let $G_1: \mathcal{A}^\N\to \{-1,0,1\}, G_1(x)=g_1(x_0,x_1)$. The associated tree $\mathcal{T}_{G_1}$ is an infinite tree as shown in Figure \ref{Tree1}.
		\begin{figure}[h]
			\centering
			\includegraphics[width=0.5\textwidth]{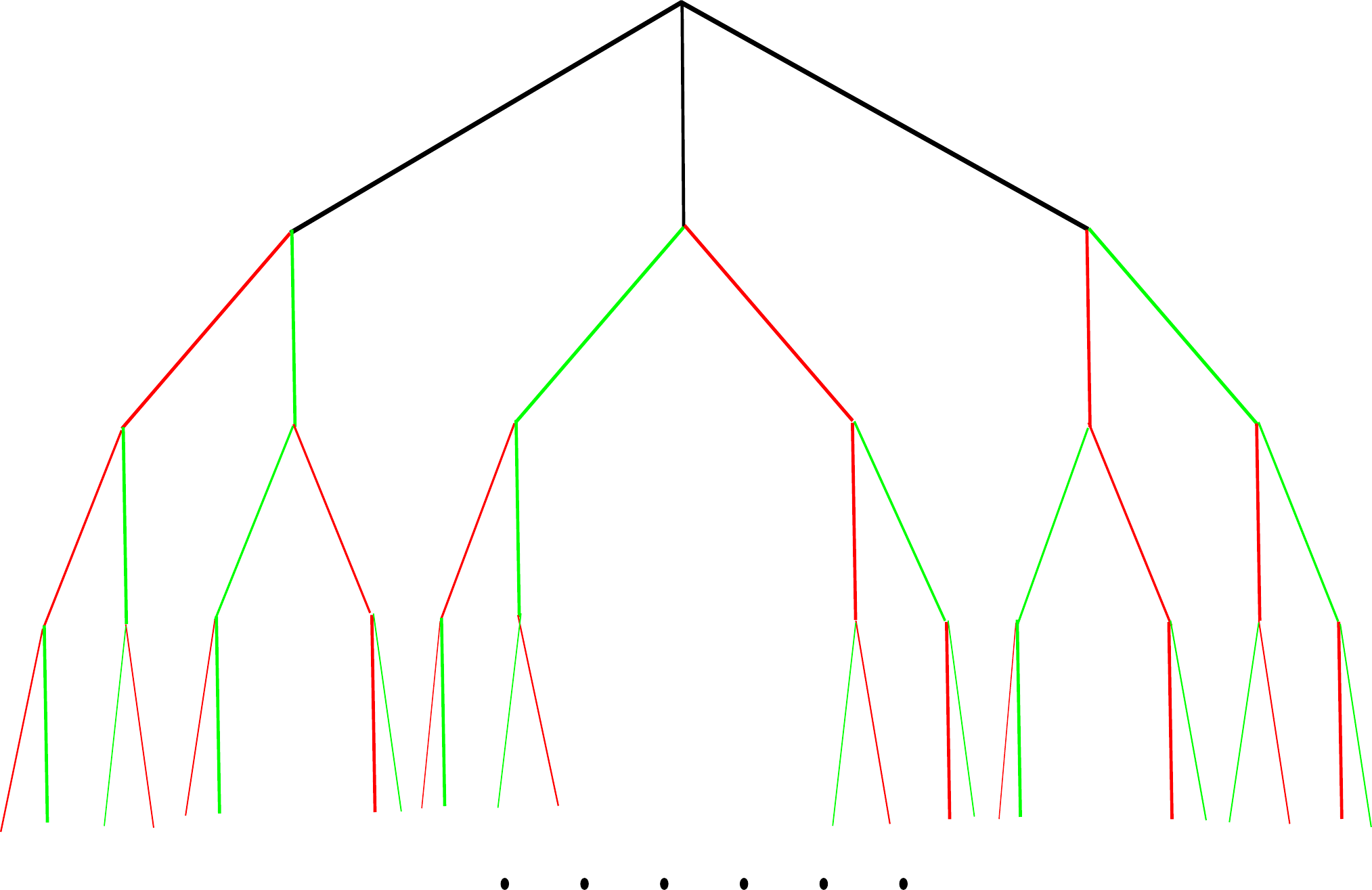}\\
			\caption{The infinite tree $\mathcal{T}_{G_1}$.}
			\label{Tree1}
		\end{figure}
		
		\item[(2)] Let $\mathcal{A}=\{0,1,2,3\}$. Define $g_1: \mathcal{A}^2\to \{-1,0,1\}$ by 
		\vspace{3pt}
		$$
		(0, 1)\mapsto 1, (1, 3)\mapsto -1, (2, 0)\mapsto -1 ~\text{and otherwise}~(a,b)\mapsto 0.
		$$ 
		\vspace{3pt}
		Let $G_2: \mathcal{A}^\N\to \{-1,0,1\}, G_2(x)=g_2(x_0,x_1)$.
		The associated tree $\mathcal{T}_{G_2}$ is a finite tree as shown in Figure \ref{Tree2}.
		\begin{figure}[h]
			\centering
			\includegraphics[width=0.5\textwidth]{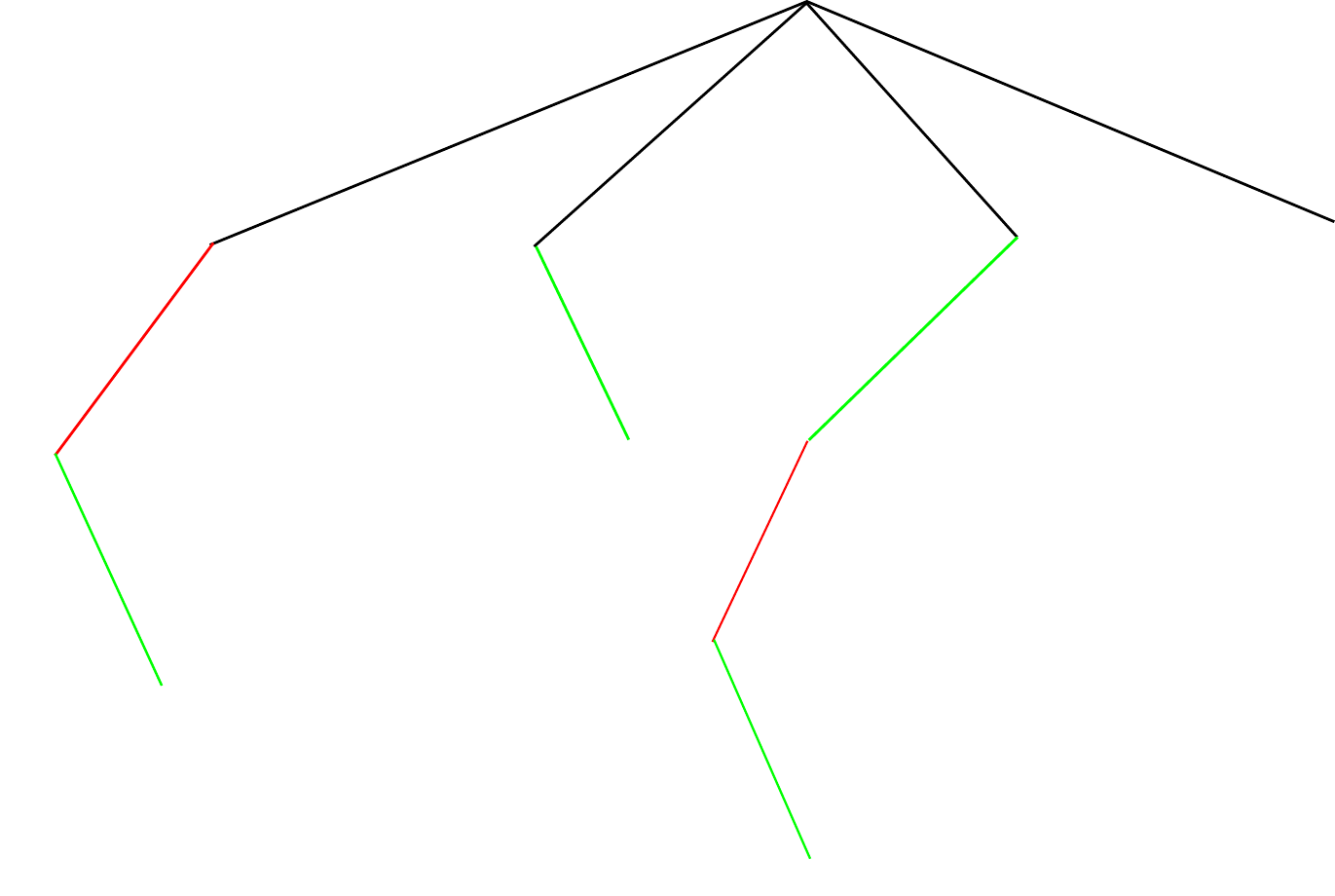}\\
			\caption{The finite tree $\mathcal{T}_{G_2}$.}
			\label{Tree2}
		\end{figure}		
	\end{itemize}
\end{ex}

Two branches from the root to the vetrices of the same level are said to be \textit{of different type} if they have different colors in some edges of the some level. We define the homogeneity for the tree as follows. Two branches are \textit{of same type} if they are not of different type.
\begin{df}\label{Def: Homogeneous}
	Suppose the function $G:\mathcal{A}^\N\to \{-1,0,1\}$ is decided by the words of length $l$. We say that the associated tree $\mathcal{T}_G$ is \textbf{homogeneous} if for each $k\ge l$, the numbers of branches of same type from the root to the vetrices at the $k$-th level are constant.
\end{df}
It is not hard to see that Example \ref{example} $(1)$ is homogeneous but $(2)$ is not.

Actually, Definition \ref{Def: Homogeneous} means that  
\begin{align*}
&\text{Card}\{x\in \mathcal{A}^{\N}: G(x)=c_0, \dots, G(S^kx)=c_k \}
\end{align*}
is a constant which is independent of the choice of $c=(c_n)_{n\in \N}\in \{-1,1\}^{\N}$. Thus the following proposition is a direct consequence of Proposition \ref{Prop: StationaryDynamical}.

\begin{prop}\label{Prop: homogeneous tree}
	Let $A$ be a finite set. Let $G: A^\mathbb{N} \to \{-1,0,1\}$ be a continuous map. Then the random sequence $(Z_{G,n})_{n\in \N}$ is a Chowla sequence almost surely if and only if the associated tree $\mathcal{T}_G$ is homogeneous.
\end{prop}

Now we show the existence of sequences of dependent random variables having Chowla property almost surely in the following example.
\begin{ex}
	Let $\mathcal{A}=\{-1,0,1, \star \}$. Denote by $\widetilde{\mu}_2=\frac{1}{2}\mu_2+\frac{1}{2}\delta_{\star}$. Then $(\mathcal{A}^\mathbb{N}, S, \widetilde{\mu}_2^\mathbb{N})$ is a measurable dynamical system. Define $g: \mathcal{A}^2 \to U(m)$ by 
	$$
	g(a,b)=
	\begin{cases}
	a~&\text{if}~a\not= \star~\text{and}~b=\star;\\
	0&\text{otherwise}.
	\end{cases}
	$$
	We define $G: \mathcal{A}^\mathbb{N} \to U(m)$ by $(x_n)_{n\in \N} \mapsto g(x_0x_1)$. It is not hard to see that the associated tree $\mathcal{T}_{G}$ is a finite homogeneous tree (see Figure \ref{Tree222}). It follows directly from Proposition \ref{Prop: homogeneous tree} that almost surely $(Z_{G,n})_{n\in \N}$ 
	is a Chowla sequence.
	
	On the other hand, since
	$
	\mathbb{E}[|Z_1|]=\frac{1}{2},
	$
	we obtain that 
	$$
	\mathbb{E}[|Z_1Z_2|]=0\quad\text{but}\quad\mathbb{E}[|Z_1|]\mathbb{E}[|Z_2|]=\frac{1}{4},
	$$
	that is,
	$$
	\mathbb{E}[|Z_1Z_2|]\not=\mathbb{E}[|Z_1|]\mathbb{E}[|Z_2|].
	$$
	By Corollary \ref{Cor: Ch then independent}, the sequence of random variables $(Z_n)$ is not independent.
\end{ex}

\begin{figure}[h]
	\centering
	\includegraphics[width=0.5\textwidth]{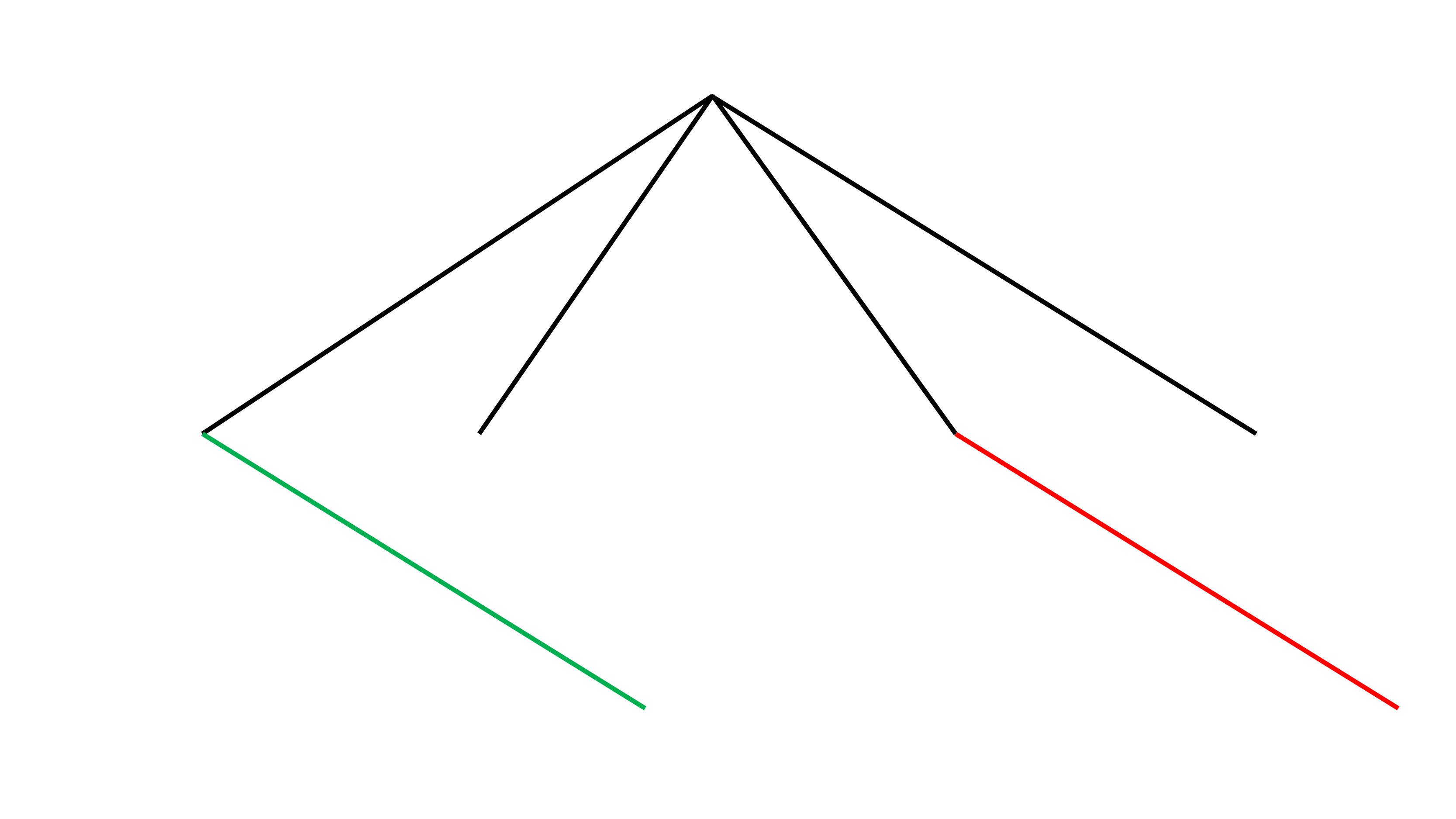}\\
	\caption{The associated tree $\mathcal{T}_{G}$}
	\label{Tree222}
\end{figure}

Let $\mathcal{A}$ be a finite alphabet. Let $(\mathcal{A}^{\mathbb{N}}, S, \mu)$ be a symbolic dynamics where $\mu$ is the uniform-Bernoulli measure. Let $m\in \N_{\ge 2}$. Let $G:\mathcal{A}^{\mathbb{N}}\to U(m)\cup \{0 \}$ be a continuous function with $G^{-1}(t)$ not empty for each $t\in U(m)\cup \{0 \}$. Denote by $Z_{G,n}:=G\circ S^n$. Similarly to the case when $m=2$, we associate a rooted $m+1$-color tree $\mathcal{T}_G$ with the sequence $(Z_{G,n})_{n\in \N}$. Similarly to Definition \ref{Def: Homogeneous} and Proposition \ref{Prop: homogeneous tree}, we define the homogeneity of the associated tree  $\mathcal{T}_G$ and have the following proposition.
\begin{prop}\label{Prop: homogeneous tree m}
	Let $A$ be a finite set and $m\ge 2$ be an integer. Let $G: A^\mathbb{N} \to U(m)\cup\{0\}$ be a continuous map. Then the random sequence $(Z_{G,n})_{n\in \N}$ is a Chowla sequence almost surely if and only if the associated tree $\mathcal{T}_G$ is homogeneous.
\end{prop}
For $m\in \N_{\ge 2}$, we give examples of sequences of dependent random variables having the index $m$ and having Chowla property almost surely. 
\begin{ex}
	Let $m\in \N_{\ge 2}$. Let $\mathcal{A}=U(m)\cup \{0, \star \}$. Denote by $\widetilde{\mu}_m=\frac{1}{2}\mu_m+\frac{1}{2}\delta_{\star}$. Then $(\mathcal{A}^\mathbb{N}, S, \widetilde{\mu}_m^\mathbb{N})$ is a measurable dynamical system. Define $g: \mathcal{A}^2 \to U(m)$ by 
	$$
	g(ab)=
	\begin{cases}
	a~&\text{if}~a\in U(m)~\text{and}~b=\star;\\
	0&\text{otherwise}.
	\end{cases}
	$$
	We define $G: \mathcal{A}^\mathbb{N} \to U(m)$ by $(x_n)_{n\in \N} \mapsto g(x_0x_1)$. It is not hard to see that the associated tree $\mathcal{T}_{G}$ is a finite homogeneous tree. It follows directly from Proposition \ref{Prop: homogeneous tree m} that almost surely $(Z_{G,n})_{n\in \N}$ 
	is a Chowla sequence.
	
	On the other hand, since
	$
	\mathbb{E}[|Z_1|]=\frac{1}{2},
	$
	we obtain that 
	$$
	\mathbb{E}[|Z_1Z_2|]=0~\text{but}~\mathbb{E}[|Z_1|]\mathbb{E}[|Z_2|]=\frac{1}{4},
	$$
	that is,
	$$
	\mathbb{E}[|Z_1Z_2|]\not=\mathbb{E}[|Z_1|]\mathbb{E}[|Z_2|].
	$$
	By Corollary \ref{Cor: Ch then independent}, the sequence of random variables $(Z_n)$ is not independent.
\end{ex}

\subsection{Sarnak property does not imply Chowla property}\label{Sec: S does not imply Ch}

In this subsection, for every $m\in \mathbb{N}_{\ge 2}\cup \{0\}$, we prove that Sarnak property does not imply Chowla property for sequences of index $m$. We first prove a sufficient condition for sequences to be Sarnak sequences.
\begin{lem}\label{Lem: K system S sequence}
	Let $(X, T, \mu )$ be a Kolmogorov system. Let $f: X \to \mathbb{C}$ be a measurable function. Suppose that $\mathbb{E}[f]=0$. Then for every generic point $x\in X$ of $\mu$, the sequence $(f\circ T^n(x))_{n\in \N}$ is a Sarnak sequence.
\end{lem}
\begin{proof}
	Let $(Y, U)$ be an arbitrary topological dynamical system. Let $y\in Y$ be a completely deterministic point. Let $x$ be a generic point for $\mu$. Let $\rho$ be a quasi-generic measure for $(y,x)$ along $(N_k)_{k\in \N}$ in the product system $(Y\times X, U\times T)$.
	It follows that for any $g\in C(Y)$,
	\begin{equation}\label{eq:K system S sequence 1}
	\lim\limits_{k\to +\infty} \sum_{n=0}^{N_k-1} g\circ U^n(y)\cdot f\circ T^n (x)=\mathbb{E}^\rho[g\otimes f].
	\end{equation}
	Since the systen $(Y, U, \rho_{|Y})$ is of zero entropy and $(X, T, \mu )$ is a Kolmogorov system, $(X, T, \mu )$ and $(Y, U, \rho_{|Y})$ are independent. It follows that for any $g\in C(Y)$,
	\begin{equation}\label{eq:K system S sequence 2}
	\mathbb{E}^\rho[g\otimes f]=\mathbb{E}^{\rho_{|Y}}[g]\cdot\mathbb{E}^\mu[f]=0.
	\end{equation}
	Combining (\ref{eq:K system S sequence 1}) and (\ref{eq:K system S sequence 2}), we deduce that for any $g\in C(Y)$,
	$$
	\lim\limits_{k\to +\infty} \sum_{n=0}^{N_k-1} g\circ U^n(y)\cdot f\circ T^n(x) =0.
	$$
	This completes the proof.
\end{proof}

Comparing the condition in Lemma \ref{Lem: K system S sequence} and the necessary and sufficient condition for sequences of stationary random variables to be Chowla sequences almost surely in Proposition \ref{Prop: StationaryDynamical}, we prove the existence of the sequences which have Sarnak property but not Chowla property. We illustrate this in the following examples.

\begin{ex}
	Recall that a Bernoulli system is always a Kolmogorov system. Let $(\{0,1,2\}^\mathbb{N}, S, B_3)$ be a Bernoulli system where $B_3$ denotes the uniform-Bernoulli measure. Let $g: \{0,1,2\}^2 \to \{-1, 0,1\}$ be a function defined by 
	$$g(0,1)=1, g(1,2)=-1 ~\text{and}~ g(a,b)=0 ~\text{otherwise}.$$ Define $G: \{0,1,2\}^\mathbb{N} \to \{-1, 0,1\}$ by $(x(n))_{n\in \N} \mapsto g(x(0),x(1))$. Let $Z_n:=G\circ S^n$. It is easy to check that 
	$$
	\mathbb{E}[Z_1]=0\quad\text{but}\quad\mathbb{E}[Z_1Z_2]=-1.
	$$
	By Lemma \ref{Lem: K system S sequence}, it follows that almost surely $(Z_n)_{n\in \N}$ is a Sarnak sequence. On the other hand, by Proposition \ref{Prop: StationaryDynamical}, almost surely the sequence $(Z_n)_{n\in \N}$ is not a Chowla sequence.
\end{ex}

\begin{ex}
	Define $h: [0,1)\to [0,1)$ by
	$$
	h(t)=
	\begin{cases}
	\frac{3}{2}t ~&\text{if}~0\le t<\frac{1}{4},\\
	\frac{1}{2}t+\frac{1}{4} ~&\text{if}~\frac{1}{4}\le t<\frac{1}{2},\\
	\frac{3}{2}t-\frac{1}{4} ~&\text{if}~\frac{1}{2}\le t<\frac{3}{4},\\
	\frac{1}{2}t+\frac{1}{2} ~&\text{if}~\frac{3}{4}\le t<1.\\
	\end{cases}
	$$
	We define the function $H: S^1\to S^1$ by the formula $H(e^{2\pi i t})=e^{2\pi ih(t)}$. Let $((S^1)^\N, S, Leb^{\otimes \N})$ be a Kolmogorov system where $Leb$ is the Lebesgue measure on $S^1$. Let $Z_n:=H\circ F\circ S^n$. It is not hard to see that almost surely the index of the sequence $(Z_n)_{n\in \N}$ is $\infty$.
	A simple computation allows us to see that
	$$
	\mathbb{E}[Z_1]=0\quad\text{but}\quad\mathbb{E}[Z_1^2]\not=0.
	$$
	By Lemma \ref{Lem: K system S sequence} and Proposition \ref{Prop: StationaryDynamical}, we conclude that almost surely $(Z_n)_{n\in \N}$ is a Sarnak sequence but not a Chowla sequence.
\end{ex}

\section*{Acknowledgement}
We would like to thank Prof. Shigeki Akiyama for the helpful discussion. We also thank Prof. El Abdalaoui Houcein El Abdalaoui for bringing our attention to the
work of \cite{Aba2017} and \cite{AbaKulLemDe2017}.

\end{document}